\documentclass[11pt]{article}
\usepackage{amsmath}%
\usepackage{amsthm}%
\usepackage{amsfonts}%
\usepackage{amssymb}%
\usepackage{amsxtra}
\usepackage{graphicx}
\usepackage{diagrams}
\usepackage{tikz}
\usetikzlibrary{arrows}
\usepackage{tkz-graph}
\usepackage{tkz-berge}

\usepackage{diagrams}
\newcommand{\rad}{\operatorname{rad}}

\newcommand{\rank}{\operatorname{rank}} 
\newcommand{\corank}{\operatorname{corank}}

\newcommand{\Ker}{\operatorname{Ker}}
\newcommand{\Image}{\operatorname{Im}}

\newcommand{\Plucker}{\operatorname{Pl\ddot{u}cker}}

\newcommand{\abs}[1]{\left|\, #1 \, \right|}

\newcommand{\Sing}{\operatorname{Sing}}

\newcommand{\Sym}{\operatorname{Sym}}

\newcommand{\Gr}{\operatorname{Gr}}

\newcommand{\Corank}{\operatorname{Corank}}

\newcommand{\Hom}{\operatorname{Hom}}

\newcommand{\codim}{\operatorname{codim}}

\newcommand{\val}{\operatorname{val}}
\newcommand{\I}{\operatorname{\mathcal{I}}}

\newcommand{\V}{\operatorname{\mathcal{V}}}
\newarrow{Equal}=====

\newcommand{\Coker}{\operatorname{Coker}}

\newcommand{\Mult}{\operatorname{Mult}}

\usepackage{url}
\usepackage[bookmarks=true]{hyperref}
\newtheorem{theorem}{Theorem}[section]
\newtheorem{corollary}[theorem]{Corollary}
\newtheorem{lemma}[theorem]{Lemma}
\newtheorem{proposition}[theorem]{Proposition}

\theoremstyle{remark}
\newtheorem{remark}[theorem]{Remark}
\newtheorem{example}[theorem]{Example}

\theoremstyle{definition}
\newtheorem{definition}[theorem]{Definition}

\begin{document}
\title{On the Singular Structure of Graph Hypersurfaces}
\author{Eric Patterson}
\maketitle

\begin{abstract}
We show that the singular loci of graph hypersurfaces correspond set-theoretically to their rank loci.
The proof holds for all configuration hypersurfaces and depends only on linear algebra.
To make the conclusion for the second graph hypersurface, 
we prove that the second graph polynomial is a configuration polynomial.
The result indicates that there may be a fruitful interplay between the current research in
graph hypersurfaces and Stratified Morse Theory.
\end{abstract}

\section{Introduction}
Let $G$ be a graph, $E$ its set of edges, and $V$ its set of vertices.
For each edge $e\in E$, let $A_e$ be an associated variable. 
The first graph polynomial is
\[
\Psi_G(A) = \sum_{F\subset G} \prod_{e\notin F} A_e
\]
where the sum is over the spanning forests of $G$.
A spanning forest is defined in Definition~\ref{D:spanningForest}
to be a subgraph of $G$ that is a spanning tree on each
component of $G$.
Similarly, the second graph polynomial is
\[
\Phi_G(p,A) = \sum_{T\subset G} s_T(p) \prod_{e\notin T} A_e 
\]
where the sum is over the quasi-spanning forests $T$ of $G$, $p\in K^{V,0}$ is a momentum of $G$, and
$s_T(p)$ is the momentum of $p$ on $T$ (see Definitions~\ref{D:quasiSpanningForest},~\ref{D:momentum}, and~\ref{D:qsfMomentum}). 
In the physics and graph theory literature,
these polynomials are also called Symanzik polynomials and Kirchoff polynomials.
They appear in parametric Feynman integrals in quantum field theory of the form 
\begin{equation} \label{eq:FeynmanIntegral}
I_G(p)= \int_{\sigma} dA\frac{\, e^{i\frac{\Phi_G(p,A)}{\Psi_G(A)}}\,}{\Psi_G(A)^2},
\end{equation}
where $\sigma$ is the domain of real values of the variables $A_e$ that sum to $1$ (see~\cite{IZ} or~\cite{BW}).
Broadhurst and Kreimer~\cite{BroadhurstKreimer95,BroadhurstKreimer97} 
related the values of $I_G(p)$ to multiple zeta values in many cases 
where the momentum $p$ factors out
of the integral trivially.

The results of Broadhurst and Kreimer motivated researchers to understand
the correspondence between graphs and multiple zeta values via motives. 
For the case of where $p$ factors out of $I_G(p)$, Bloch, Esnault, and Kreimer~\cite{BEK}
have interpreted equation~\eqref{eq:FeynmanIntegral}
as the period of the mixed Hodge structure of the graph hypersurface $X_G$ defined by $\Psi_G(A)=0$. 
In particular, the middle dimensional cohomology of the 
graph hypersurface needs to be computed.

For the case of wheel-and-spokes graphs, Bloch, Esnault, and Kreimer have
produced such a computation in~\cite{BEK}, and Doryn~\cite{DD} has extended
their techniques to zig-zag graphs.
Aluffi and Marcolli~\cite{AluffiMarcolli09} show that the
mixed Hodge structure for three loop graphs is mixed Tate,
which suggests that the period could be a multiple zeta value.
Brown~\cite{Brown09b} derived the mixed Tate conclusion
for a whole class of graphs for which he was also able
to prove that the Feynman integrals evaluate to multiple zeta values.
Brown's results provide an explanation for the known analytic computations
of Feynman integrals and suggest where to look for non-multiple zeta values.

Our goal is to extend the toolbox of cohomological tools
available to connect Feynman integrals and multiple
zeta values.
To this end, we prove some simple results on the structure
of the singular loci of graph hypersurfaces. 
We hope that these formulas may provide some tools 
for applying Stratified Morse Theory~\cite{StratifiedMorseTheory} to the
study of the periods of mixed Hodge structures of graph hypersurfaces.

In particular, we are able to identify the multiplicity
of the points in the graph hypersurfaces as specific rank
loci.
More specifically, following Proposition~2.2 in~\cite{BEK}, the first graph polynomial
is the determinant of a certain canonical bilinear form on $K^{E}$
restricted to the subspace $H_1(G,K)$. 
Here $K^E$ is 
the vector space over a field $K$ with basis $E$,
and $H_1(G,K)$ is the first homology of the graph $G$ with
coefficients in $K$.
Similarly, Proposition~\ref{P:secondGraphPoly} shows that the second graph polynomial
is the determinant of the same canonical bilinear form on $K^{E}$
but restricted to a relative homology subspace $H_1(G,p)$ 
defined in Definition~\ref{D:relativeHomology}.
Therefore, these polynomials can be analyzed with the linear
algebra of \emph{configurations} and \emph{configuration polynomials}.

The main result, Theorem~\ref{T:singularityRank}, proves that the multiplicity of a point
in the hypersurface defined by a configuration polynomial
is one less than its corank as a bilinear form.
In particular, the theorem holds for the first and second graph 
hypersurfaces by Proposition~2.2 in~\cite{BEK} and Proposition~\ref{P:secondGraphPoly} here.

The approach via configuration polynomials that we present in this paper
is complementary to the usual approach to graph polynomials
via the graph Laplacian (see the review~\cite{BW}).
In particular, the graph polynomials may be derived as
determinants of submatrices of the graph Laplacian, which is
the content of the all-minors matrix-tree theorem.
The graph Laplacian is
defined in terms of the vertices of the graph, and
the all-minors matrix-tree theorem must verify that
certain choices of submatrix do not affect the computation.
For comparison, the configuration approach
lifts the computation from the vertices to the edges
via the homology boundary map.
The choice of submatrix in the graph Laplacian approach
becomes the choice of
a representing element in a homology group for
the configuration approach.
This analogy suggests that the most general form
of the all-minors matrix-tree theorem
may have an analogous expression
in terms of configuration polynomials.

The configuration hypersurfaces naturally map to the generic
symmetric determinantal varieties. 
Moreover, the generic symmetric determinantal varieties satisfy
the same relationship between the corank and the multiplicity
of their points (e.g., see Theorem~22.33 in~\cite{Harris}).
The basic idea behind applying Stratified Morse Theory suggested
by Bloch~\cite{BlochSlides} is to find a stratified Morse function
for the generic symmetric determinantal variety relative to its rank stratification,
which could draw on a host of well-known results about the geometry of these
varieties. 
Provided  such a function, its restriction to 
the configuration hypersurface would be a stratified Morse function,
and 
we hope that our result, by identifying the strata by their multiplicity,
could provide some ideas for
how to generate new results on the topology of the configuration 
hypersurface.
We also note that Brown~\cite{Brown09b} uses Stratified Morse Theory
following a different approach than the one just mentioned.

Among applications that we can prove, we provide a formula for the tangent
cones of a graph hypersurface over an algebraically closed field of
characteristic zero
in Proposition~\ref{P:configurationTangentCone}.
In addition, $H_1(G)$ is a subspace of $H_1(G,p)$, which provides
some further relationships between the first and second graph
hypersurfaces, for which we refer to~\cite{MyThesis}.
For some applications, 
one may need to strengthen our results
from the perspectives of
both algebraic geometry and physics.
Our identification of multiplicity loci with degeneracy
loci does not make the identification at the scheme level,
and we only work in the context of a scalar field.
Bloch and Kreimer show how to extend the configuration definition
of the second graph polynomial
from a scalar field to quaternions in~\cite{BK10}.

We begin in Section~\ref{S:bilinearForms} with
some elementary results on bilinear forms.
Section~\ref{S:configurations} recalls the
definitions of configurations, configuration polynomials,
and how graph polynomials fit into these definitions.
Section~\ref{S:hypersurfaces} proves the main
multiplicity-corank correspondence and its consequences.
We conclude in Section~\ref{S:conclusion}.

We would like to thank Spencer Bloch for his invaluable
guidance during this research and
continued encouragement to provide these results to the 
mathematical community.

\section{Symmetric Bilinear Forms}\label{S:bilinearForms}
This section has two purposes.
First, it reviews the terminology and notation on symmetric bilinear forms that we will use.
Second, it proves a fundamental lemma about degenerate symmetric bilinear forms: Lemma~\ref{L:lowRankSubspace}.
The results generalize to nonsymmetric bilinear forms, but there are no uses of that generality in this paper.
The reader familiar with degenerate bilinear forms may prefer to skip this section.
The reader seeking more details should consult~\cite{MyThesis}.

\subsection{Notation and Basic Definitions}
We assume all vector spaces are finite dimensional.
Let $V$ be a vector space over a field $K$. Denote the dual space by $V\spcheck=\Hom_K(V,K)$.
A \emph{bilinear form} $B$ on $V$ is an element of 
$(V\otimes V)\spcheck= V\spcheck\otimes V\spcheck$.
A bilinear form $B$ is \emph{symmetric} if $B(v\otimes w)= B(w\otimes v)$ for all $v$, $w\in V$.

A symmetric bilinear form $B$ defines a linear map from $V$ to $V\spcheck$:
\begin{align*}
0 \to \rad (V,B) \to  V&\xrightarrow{\ell_B}  V\spcheck \\
\ell_B(v)&=B(v,\cdot)=B(v\otimes \cdot).
\end{align*}
\begin{definition}[Radical]
The \emph{radical of $B$ in $V$}, $\rad (V,B)$, is 
\[
\rad (V,B) = \Ker \ell_B=\{ v\in V \, |\, B(v,w)=0 \mbox{ for all }w\in V\}.
\] 
\end{definition}
\begin{definition}[Degenerate]\label{D:degenerateBilinearForm}
The bilinear form $B$ is called \emph{degenerate} if $\rad (V,B)=0$.
\end{definition}
We will often work with one symmetric bilinear form $B$ and restrict it to various subspaces $W$ of $V$. In that case, we may suppress
the $B$ from the notation and write $\rad V$ for $\rad (V,B)$, $\rad W$ for $\rad (W,B|_W)$.

The rank of $B$ on $V$ is the rank of the linear map $\ell_B$, and the corank of $B$ on $V$ is 
\[
\corank B = \dim \rad (V,B) = \dim V-\rank B.
\]
When we restrict
$B$ to a subspace $W$, we will let
$\rank_W B$ and $\corank_W B$
denote $\rank\left( B|_W\right)$ and $\corank\left( B|_W\right)$. In particular,
$\rank_V B$ and $\corank_V B$ are synonyms for $\rank B$ and $\corank B$ that include $V$ in the notation.
The following lemma is well-known, and we will not include the proof. 
We make frequent use of the equivalence of these conditions, often without comment. 
\begin{lemma}
Suppose that $B$ is a symmetric bilinear form on $V$. 
The following conditions on $B$ are equivalent:
\begin{enumerate}
\item $\corank B > \dim V -k$,
\item $\rank B < k$, and
\item all $k\times k$ minors of $M$ vanish for a matrix $M$ representing $B$
in some basis of $V$.
\end{enumerate}
\end{lemma}

A \emph{hyperplane} in $V$ is a linear subspace of $V$ of codimension one (i.e., we are 
considering hyperplanes in the linear subspace sense, not the affine sense).
\begin{definition}
A set of hyperplanes $\{H_1,\dots,H_k\}$ in $V$ is \emph{complete} if $\cap_{i=1}^k H_i =0$.
Note that $k$ may be greater
than the dimension of $V$ in this definition.
\end{definition}
Every hyperplane is the kernel of a linear form on $V$, and scaling the linear form does not change the kernel. 
Therefore,
each hyperplane $H$ in $V$ is identified with a point $H\spcheck$ in $\mathbb{P}(V\spcheck)$.
We mention this well-known fact so the following proposition makes sense.
\begin{proposition}[Complete Equivalent to Spanning]
A set of hyperplanes $\{ H_1,\dots, H_k\}$ in $V$ is complete if and only if $\{ H_1\spcheck,\dots,H_k\spcheck\}$ span $V\spcheck$.
\end{proposition}
This fact is well-known, so we will not repeat the proof.
Spanning is probably the preferred adjective,
but we prefer the term complete to emphasize that
we will use the intersection property. 

\subsection{Fundamental Lemmas}
The main result of this section is Lemma~\ref{L:lowRankSubspace}, which translates
bounds on the rank of a degenerate symmetric bilinear form restricted to a complete set of hypersurfaces
to a bound on the rank of the unrestricted bilinear form.
It is the key lemma in Theorem~\ref{T:singularityRank}, which identifies the multiplicities
of configuration hypersurfaces with rank loci.

Let $V$ be a vector space, and let $B$ be a symmetric bilinear form on $V$.
For every subspace $W$ of $V$, the radicals of $B$ fit into the following diagram:
\begin{equation}\label{D:basicRadDiagram}
\begin{diagram}
0 & \rTo & \rad W & \rInto & W & \rTo^{\ell_{(B|_W)}} & W\spcheck \\
   &   &  &           & \dInto & \rdTo^{(\ell_B)|_W} & \uOnto \\
0 & \rTo & \rad V  & \rInto & V & \rTo^{\ell_B}  & V\spcheck. \\
\end{diagram}
\end{equation}
The rows are left exact by the definition of the radical.
A key point to note is that there is no natural map between the radicals of $V$ and $W$. 

\begin{lemma} \label{L:radicalIntersection}
Let $B$ be a symmetric bilinear form on a vector space $V$, and let $W$ be a subspace of $V$. Then
\[
W \cap \rad V \subseteq \rad W.
\]
\end{lemma}
\begin{proof}
Consider an element $w \in W\cap \rad V $. As an element of $\rad V $, $B(v,w)=0$ for
all $v \in V$ and, in particular, for all $v\in W\subseteq V$. Because $w$ is in $W$, this proves that $w\in \rad W$.

Alternatively, note that $W\cap \rad V$ is the kernel of the diagonal map, $ (\ell_B)|_W$, in the diagram~\eqref{D:basicRadDiagram}. 
This kernel is a subspace of $\rad W$, which is a consequence of the commutativity
of the diagram. 
\end{proof}
\begin{corollary} \label{C:radicalContainment}
If $\rad V\subseteq W$, then $\rad V \subseteq \rad W $.
\end{corollary}
Let $S$ be a subset of $V$, and define the \emph{subspace orthogonal to $S$},
\[
\bot(S) = \{ v\in V \, |\, B(v,s)=0 \mbox{ for all } s\in S \}.
\]
For every set $S$, $\bot(S)$ contains $\rad V$.
Note that $W$ is always contained in $\bot(\rad W)$.
Therefore,
\begin{lemma}\label{L:orthogonalRadical}
For every subspace $W$ of $V$,
\[
W+\rad V \subseteq \bot( \rad W ).
\]
\end{lemma}
\begin{corollary}\label{C:subspaceRadSpan}
If $W+\rad V = V$, then $\rad W \subseteq \rad V$.
\end{corollary}
\begin{proof}
By Lemma~\ref{L:orthogonalRadical}, 
\[
	V=W+\rad V \subseteq \bot(\rad W),
	\]
which implies $V=\bot(\rad W )$.
In other words, $B(v,w)=0$ for every $v\in V$ and every $w\in \rad W$, 
which means $\rad W\subseteq \rad V$.
\end{proof}
\begin{corollary}\label{C:hyperplaneLemma}
If $H$ is a hyperplane in $V$, then $\rad V \subseteq \rad H$ or $\rad H\subseteq \rad V$.
\end{corollary}
\begin{proof}
If $\rad V \subseteq H$, then $\rad V \subseteq \rad H$ by Corollary~\ref{C:radicalContainment}.
If $\rad V\nsubseteq H$, then $H+\rad V=V$ because $H$ is a hyperplane. 
Then Corollary~\ref{C:subspaceRadSpan} implies that $\rad H\subseteq \rad V$.
\end{proof}
\begin{lemma} \label{L:rankDimensionDegenerate}
If $B$ is degenerate on $V$, then every complete set of hyperplanes has an element $H$ for which 
$\rank_V B = \rank_H B$. 
\end{lemma}
\begin{proof}
If $\rad V$ is contained in every hyperplane in a complete set, 
then $\rad V$ is contained in their intersection, which is zero.
However $B$ is degenerate, so $\rad V$ must be nonzero, and therefore, 
there is at least one hyperplane $H$ in
the complete set for which $\rad V\nsubseteq H$.
By Corollaries~\ref{C:radicalContainment} and~\ref{C:hyperplaneLemma}, $\rad H\varsubsetneq \rad V$.
Therefore, $\corank_H B < \corank_V B$, and thus,
\begin{align*}
n-1-\rank_H B &< n- \rank_V B \\
\rank_H B &> \rank_V B -1\\
\rank_H B &\geq \rank_V B.
\end{align*}
The ranks are equal to the dimensions of the images of $\ell_{(B|_H)}$
and $\ell_B$, respectively. The factorization of $\ell_{(B|_H)}$
depicted in diagram~\eqref{D:basicRadDiagram} implies that
\[
\dim \Image \ell_{(B|_H)} \leq \dim \Image \ell_B;
\]
hence $\rank_H B = \rank_V B$.
\end{proof}

\begin{lemma}[Criterion to Bound Rank] \label{L:lowRankSubspace}
Suppose that $B$ is a  degenerate symmetric bilinear form on $V$ and $k$ is an integer, $1\leq k< \dim V$.
Let $\{H_1,\dots,H_{\ell}\}$ be a complete set of hyperplanes in $V$.
For a nonempty subset $J\subseteq \{ 1,\dots, \ell\}$, define
\[
H_J = \cap_{j\in J} H_j.
\]
If $B|_{H_J}$ is degenerate for all subsets $J$ with $\abs{J}\leq k$, then
\[
\rank_V B < \dim V -k.
\]
\end{lemma}
\begin{proof}
Proceed by induction on $\dim V$. For each step of the induction, we must verify the
lemma for the whole set of positive integers $k$ less than $\dim V$.
If $\dim V=1$, there are no positive integers $k$ less than $\dim V$
and no nonzero hyperplanes, so the statement is empty and there is nothing to prove.

If $\dim V=2$, then $k=1$ is the only positive integer less than $\dim V$. 
By Lemma~\ref{L:rankDimensionDegenerate}, there is a hyperplane (i.e., a line) $H_j$ in the complete set
for which
\[
\rank_V B = \rank_{H_j} B.
\]
By assumption, $B|_{H_j}$ is degenerate, so $\rank_{H_j} B < \dim H_j = 1$. That is, $\rank_V B =0$, which is less than
$n-k=2-1=1$.

Suppose the lemma is true for vector spaces of dimension $n-1$, and suppose $\dim V =n$.
By Lemma~\ref{L:rankDimensionDegenerate}, there is a hyperplane in the complete set, say $H_1$,
for which
\[
\rank_V B = \rank_{H_1} B.
\]
The inductive hypothesis applies to the hyperplane $H_1$ and the complete set of hyperplanes in $H_1$:
\[
	\{ H_1\cap H_j\, |\, j=2,\dots,\ell\mbox{ and } H_1\neq H_j \}.
\]
By assumption, $B|_{H_J}$ is degenerate for all subsets $J\subseteq \{1,\dots,\ell\}$ with $\abs{J}\leq k$.
In particular, $\left( B|_{H_1} \right)|_{H_{\tilde{J}}}$ is degenerate for each set $\tilde{J}
\subseteq \{2,\dots, \ell\}$ with $\abs{ \tilde{J} }\leq k-1$. Then the conclusion of the inductive hypothesis is
\[
\rank_V B = \rank_{H_1} B < n-1-(k-1)=n-k.
\] 
\end{proof}

\section{Configurations}\label{S:configurations}
In~\cite{BEK}, Bloch, Esnault, and Kreimer show how the first graph polynomial is a configuration polynomial.
Although graph polynomials are our main interest, the language of configurations is useful
when results depend only on linear algebra and not the combinatorics of graphs.
In this section, we recall the language of configurations, the main formula for configuration polynomials,
and how the first and second graph polynomials are specific examples of configuration polynomials.

\subsection{Configuration Polynomials}
A \emph{based vector space} is a pair $(V,E)$ of a vector space $V$ and a preferred basis $E$.
There is an isomorphism\footnote{A morphism of based vector spaces takes
basis elements to basis elements or zero.} of based vector spaces between $V$ and $K^E$, the vector space
generated by the set $E$.
Every element $v$ of $V$ can be written uniquely as $\sum_{e\in E} v_e e$ where $v_e\in K$.
The $v_e$ are the coordinates of $v$ in the basis $E$.
Denote the coordinate functions
\begin{align*}
X_e \colon V &\to K \\
v &\mapsto v_e.
\end{align*}
Note that the whole basis $E$ is needed to define each functional $X_e$.
Squaring the coordinate functions defines symmetric, rank-one bilinear forms on $V$:
\begin{align*}
X_e^2 \colon V\otimes V &\to K\\
v\otimes w &\mapsto X_e(v)X_e(w) = v_ew_e.
\end{align*}
For convenience, let $X_e^2(v)=X_e^2(v,v)$ and, more generally,
$b(v)=b(v,v)$ for bilinear forms $b$ on $V$.
Define the natural map from $V$ to the span of these bilinear forms:
\begin{align*}
B_E\colon V &\to \Sym^2 V\spcheck \\
a=\sum_{e\in E} a_e e &\mapsto \sum_{e\in E} a_e X_e^2.
\end{align*}
Suppose that $E=\{e_1,\dots, e_n\}$ so that $a=\sum_{i=1}^n a_{e_i} e_i$, 
and let the coordinates $a_{e_i}$ be
denoted $a_i$.
Each bilinear form $B_E(a)$ is diagonal when represented in the basis $E$:
\[
B_E(a) = 
\begin{pmatrix}
a_1 & 0 & \cdots & 0 \\
0 & a_2 & \cdots & 0 \\
\vdots & \vdots & \ddots &\vdots \\
0 & 0 & \cdots & a_n
\end{pmatrix}.
\]
In particular, $\det B_E(a) = a_1a_2\cdots a_n$.  

Let $\{A_e\}_{e\in E}$ denote the coordinate functions defined by $E$ on $V$.
Note that $A_e=X_e$; the change of notation is useful for distinguishing
maps from $V$ into $\Sym^2 V\spcheck$ from elements of $\Sym^3 V\spcheck$.
When the elements of $E$ are enumerated $e_1,\dots, e_n$, write $A_i$ for $A_{e_i}$.
Polynomial functions on $V$ can be expressed in terms of $\{A_e\}_{e\in E}$,
so the function $B_E$ can be written
\[
B_E = 
\begin{pmatrix}
A_1 & 0 & \cdots & 0 \\
0 & A_2 & \cdots & 0 \\
\vdots & \vdots & \ddots &\vdots \\
0 & 0 & \cdots & A_n
\end{pmatrix}.
\]
We write $B_E(A)$ when we want to emphasize this representation
of the map $B_E$.
The family $B_E(A)$ maps $V$ into $\Sym^2 V\spcheck$,
and the ideal $\langle A_1A_2\cdots A_n\rangle$ in $K[A_1,\dots,A_n]$
defines the locus of points in $V$ whose images in $\Sym^2 V\spcheck$ are degenerate, namely the coordinate
hyperplanes for the basis $E$.
The family $B_E(A)$ is homogeneous in the coordinates
and $B_E(a)$ is nonzero for nonzero $a$,
so $B_E(A)$ can also be considered as mapping
$\mathbb{P}(V)$ to $\mathbb{P}\left( \Sym^2 V\spcheck \right)$.

\begin{definition}[Configurations] \label{D:configuration}
A \emph{configuration} $W$ is a subspace of a based vector space\footnote{A \emph{generalized configuration} is a linear map $\varphi$ from a vector space $W$ to a based vector space.
In particular, for every generalized configuration $\varphi$, the image $\varphi(W)$ is a (nongeneralized) configuration in
the based vector space.
The results in this paper extend to generalized configurations.
We do not need this generality, so we refer the interested reader to~\cite{MyThesis}}. 
\end{definition}
Each configuration $W\subseteq V$ defines a family of bilinear forms $B_E|_W\subseteq \Sym^2 W\spcheck$ parametrized
by $V$ by
composing with the natural map:
\[
B_E|_W: V  \xrightarrow{B_E}\Sym^2 V\spcheck \to \Sym^2 W\spcheck.
\]
We write $X_e |_W$ and $B_E(a)|_W$ when
we want to emphasize that we have restricted to $W$. 

We can write $B_E(a)|_W$ as
a matrix in any
basis of $W$. 
When we do so,
the choice of basis for $W$ is often suppressed from the notation. 
As a matrix, $B_E(a)|_W$ is still symmetric but is not diagonal 
unless $W$ is the span of a subset of $E$.
If $\dim W=\ell$, then $B_E(a)|_W$ is $\ell\times \ell$.
As matrices with variable entries, the elements of the family $B_E|_W$ have entries linear in $A_1,\dots, A_n$.
\begin{example} \label{E:configuration1}
Let $E=\{e_1,e_2,e_3\}$, and let $W$ be the configuration spanned by $\ell_1=e_1+e_2$ and $\ell_2=2e_3-e_2$. Then 
\begin{align*}
X_1^2(\ell_1)&=1\\
X_2^2(\ell_1)&=1\\
X_3^2(\ell_1)&=0\\
X_1^2(\ell_2)&=0\\
X_2^2(\ell_2)&=1\\
X_3^2(\ell_2)&=4\\
X_1^2(\ell_1,\ell_2)&=0\\
X_2^2(\ell_1,\ell_2)&=-1\\
X_3^2(\ell_1,\ell_2)&=0.
\end{align*}
Therefore, $B_E|_W$ in the basis $\beta=\{\ell_1,\ell_2\}$ is
\[
B_E|_{W,\beta} =A_1
\begin{pmatrix}
1 & 0 \\
0 & 0
\end{pmatrix}
+
A_2
\begin{pmatrix}
1 & -1 \\
-1 & 1
\end{pmatrix}
+
A_3
\begin{pmatrix}
0 & 0 \\
0 & 4
\end{pmatrix}
=
\begin{pmatrix}
A_1+A_2 & -A_2 \\
-A_2 & A_2+4A_3
\end{pmatrix}
.
\]
An element $B_E(a)|_W$ of the family $B_E|_W$ is degenerate if and only if 
\[
\det B_E(a)|_{W,\beta}=(a_1+a_2)(a_2+4a_3)-a_2^2 = a_1a_2+4a_1a_3+4a_2a_3=0.
\]
If we write $B_E|_W$ in a different basis $\tilde{\beta}=\{\tilde{\ell}_1,\ell_2 \}$ with $\tilde{\ell}_1=3\ell_1$, 
then
\[
B_E|_{W,\tilde{\beta}}=A_1
\begin{pmatrix}
9 & 0 \\
0 & 0
\end{pmatrix}
+
A_2
\begin{pmatrix}
9 & -3 \\
-3 & 1
\end{pmatrix}
+
A_3
\begin{pmatrix}
0 & 0 \\
0 & 4
\end{pmatrix}
=
\begin{pmatrix}
9A_1+9A_2 & -3A_2 \\
-3A_2 & A_2+4A_3
\end{pmatrix}
.
\]
In particular, the determinant of a matrix representing $B_E(a)|_W$ is only defined
up to the squared determinant of the change of basis matrix as in
\begin{align*}
\det B_E(a)|_{W,\tilde{\beta}}&=(9a_1+9a_2)(a_2+4a_3)-9a_2^2 \\
&= 9a_1a_2+36a_1a_3+36a_2a_3\\
&= 9 \det B_E(a)|_{W,\beta}.
\end{align*}
This relationship is well-known, and the reader desiring further clarification
is referred to the references.
\end{example}

\begin{definition}[Configuration Polynomial] \label{D:configurationPolynomial}
Let $W$ be a nonzero configuration in $K^E$.
A \emph{configuration polynomial} of $W$, $\Psi_W(A)$, is a generator of the principal ideal $\det B_E|_W$
in $K[A_1,\dots,A_n]$. 
\end{definition} 
In particular, the determinant
of $B_E|_W$ in a basis $\beta$ for $W$ is a configuration polynomial,
and we call such a configuration polynomial \emph{the configuration determinant
for the basis $\beta$}.
Moreover, these generators are homogeneous of degree $\dim W$ as
the determinants of $\dim W$ square matrices with linear entries.
Now we will recall a more explicit formula
for $\Psi_W(A)$ using Pl\"{u}cker coordinates.

\subsection{Configuration Polynomials in Terms of Pl\"{u}cker Coordinates}
A subspace $W\subseteq V$ of dimension $\ell$ is an element of the Grassmannian of $\ell$-dimensional subspaces of $V$,
$\Gr(\ell,V)$. 
When $V$ has a preferred basis, the configuration polynomial for $W$ can be computed 
using the coordinates of the \emph{Pl\"{u}cker embedding} of $\Gr(\ell,V)$ into $\mathbb{P}\left(\bigwedge^\ell V\right)$.
The Pl\"{u}cker embedding is the map that takes an $\ell$-dimensional subspace $W$ of $V$ to the line
$\det W=\bigwedge^\ell W $ in $\bigwedge^\ell V$. 
From an ordered basis of $V$, denoted $E=\{e_1,\dots,e_n\}$, there is an induced basis on $\bigwedge^\ell V$
defined by $e_I=e_{i_1}\wedge\cdots\wedge e_{i_\ell}$ for every set $I=\{i_1,\dots,i_\ell\}$ such that
$i_1<\cdots <i_\ell$. 
Let $V_I$ denote the subspace of $V$ spanned by the $e_i$ for $i\in I$.
By picking the basis $E$, there are projections for every $I$
\begin{align*}
\pi_I: V&\to V_I \\
\sum_{e\in E} v_e e &\mapsto \sum_{e\in I} v_e e.
\end{align*}
In particular, $\det V_I$ is the line spanned by $e_I$.

The image of an $\ell$-plane $W$ under the Pl\"{u}cker embedding has projective coordinates along each $e_I$, and this set
of projective coordinates is the set of \emph{Pl\"{u}cker coordinates} of $W$.
Fix a basis $\beta$ of $W$.
For each $I$, we write $\Plucker_{I,\beta}(W)$ for
the coordinate of $\det W$ along $I$ in the basis $\beta$. 
That is, 
for an $\ell$-dimensional subspace $W$ of $V$, its coordinate $\Plucker_{I,\beta}(W)$ is the element of
$K$ such that $\det W \to \det V_I$ is multiplication by $\Plucker_{I,\beta}(W)$:
\begin{diagram}
W & \rInto & V  & & \det W &\rInto & \bigwedge^k V\\
  & \rdTo^{\pi_I|_W} &  \dOnto^{\pi_I}  & \rImplies & { } & \rdTo^{\Plucker_{I,\beta}(W)} & \dOnto\\
 & & V_I            &          &   &       & \det V_I.      \\
\end{diagram}
These coordinates are well-defined up to a constant representing a change of basis for $W$, and 
$\Plucker_{I,\beta}(W)$ is $0$ if and only if $\pi_I|_W: W\to  V_I$ is not an isomorphism.
A change of basis for $V$ changes the basis vectors $e_I$ of $\bigwedge^\ell V$, so the
coordinates of $W$ change. 
In the context of configurations, the based vector space $K^E$ has a fixed preferred basis, so 
a general change of basis is not allowed. 
The preferred basis is not assumed to be ordered, and
a change of basis of $K^E$ by reordering may change the Pl\"{u}cker coordinate for the index set $I$ by $\pm 1$.
The dependence of the coordinates on these choices of bases are suppressed from the notation.

In terms of matrices, once we have chosen the basis for $V$, we can write a basis for $W$ as $\ell$ row
vectors in the coordinates of the basis for $V$. Arranging these row vectors into a $\ell\times n$
matrix, the Pl\"{u}cker coordinates of $W$ are the $\ell\times \ell$ minors of this matrix.
The basis $E$ can be reordered to interchange columns of the $\ell\times n$ 
matrix representing $W$, which only changes the $\ell\times \ell$ minors by $\pm 1$.
Arbitrary changes of basis for $W$ changes the $\ell\times \ell$ minors by a single constant
from the change of basis matrix.
We formalize the preceding discussion in the next lemma.
\begin{lemma}[Pl\"{u}cker Coordinates are Projective Coordinates]
If $\beta$ and $\tilde{\beta}$ are bases for $W$ and $\Plucker_{I,\beta}(W)$ and $\Plucker_{I,\tilde{\beta}}(W)$
are the respective Pl\"{u}cker coordinates, then there is a constant $C$, the determinant of the change of
basis matrix from $\beta$ to $\tilde{\beta}$, such that
\[
\Plucker_{I,\beta}(W)=C\Plucker_{I,\tilde{\beta}}(W)
\]
for all $\ell$-subsets $I$ of $E$, a basis for $V$. In particular, these coordinates
are well-defined projective coordinates for the $\ell$-plane $W$.
\end{lemma}
The basis of $W$ used for the Pl\"{u}cker coordinates is a minor issue in the following,
so $\Plucker_I(W)$ is written instead of $\Plucker_{I,\beta}(W)$.
The reader may prefer to think in terms of a fixed basis though we point out 
when and how results depend on the basis.
\begin{example} \label{E:pluckerCoordinates}
Consider the configuration $W$ of Example~\ref{E:configuration1}.
In the ordered basis $E=\{e_1,e_2,e_3\}$ of $K^E$ and the basis $\{\ell_1,\ell_2\}$ of $W$, the matrix
representing $W$ is
\[
\begin{pmatrix}
1 & 1 & 0 \\
0 & -1 & 2
\end{pmatrix}.
\]
In the induced basis of $\bigwedge^2 K^E$, $\det W$ is
\begin{align*}
\det W &= \det
\begin{pmatrix}
1 & 1 \\
0 & -1
\end{pmatrix} e_1\wedge e_2
+
\det
\begin{pmatrix}
1 & 0 \\
0 & 2
\end{pmatrix} e_1\wedge e_3
+
\det
\begin{pmatrix}
1 & 0 \\
-1 & 2
\end{pmatrix}  e_2\wedge e_3 \\
&= -e_1\wedge e_2 +2 e_1\wedge e_3 + 2 e_2\wedge e_3.
\end{align*}
Therefore, its Pl\"{u}cker coordinates are $[-1:2:2]$.
Notice the relationship between the last line and the coefficients of the configuration polynomial
\[
\Psi_W(A)= A_1A_2 + 4A_1A_3 +4A_2A_3.
\]
The next proposition explains this relationship.
\end{example}
\begin{proposition}[Configuration Polynomial in Pl\"{u}cker Coordinates] \label{P:configurationPolynomial}
Let $W$ be a nonzero configuration in $K^E$.
A configuration polynomial for $W$ is
\[
\Psi_W(A)=\sum_{
	\substack{F\subset E \\ \abs{F}=\dim W} } 
	\Plucker_F(W)^2 \prod_{f\in F} A_f.
\]
\end{proposition}
\begin{remark}
The configuration polynomial is only defined up to a constant, so the ambiguity of which basis for $W$ is used
to compute  $\Plucker_F(W)^2$ on the right side is irrelevant---it will still generate the principal ideal $\det B_E|_W$.
The basis for $K^E$ is fixed up to order, which can only change the Pl\"{u}cker coordinates by $\pm 1$,
and hence the coefficients $\Plucker_F(W)^2$ are independent of the order of $E$.
Also, note that there is some nonzero Pl\"{u}cker coordinate for every nonzero configuration $W$, 
so the configuration 
polynomial is never identically zero.
\end{remark}
\begin{proof}
Pick a basis $\beta=\{w_1,\dots,w_\ell\}$ for $W$, and work with the matrix $B_E|_{W,\beta}$ representing $B_E|_W$ in this basis.
As noted above, the entries in $B_E|_{W,\beta}$ are linear in the $A_e$, so the determinant is
homogeneous of degree $\dim W$.
The bilinear form $X_e^2$ has rank at most $1$ on $W$. If it has rank $1$, then there is a basis for $W$ in which $X_e^2$ is
represented by a matrix with $0$s everywhere except a $1$ in the $(1,1)$ entry. Choosing $\beta$ to be this basis, 
\[
B_E|_{W,\beta} = A_e 
\begin{pmatrix}
1 & 0 & \dots & 0 \\
0 & 0 & \dots & 0 \\
\vdots & \vdots & \ddots & \vdots \\
0 & 0 & \dots & 0
\end{pmatrix}
+\sum_{f\in E-e} A_f X_f^2.
\]
The variable $A_e$ only appears in one entry of $B_E|_{W,\beta}$, so it appears with degree at most $1$ in $\det B_E|_{W,\beta}$.
If $X_e^2$ has rank $0$ on $W$, then 
\[
B_E|_{W,\beta} = A_e 
0+\sum_{f\in E-e} A_f X_f^2,
\]
and it is still true that $A_e$ appears with degree at most $1$ in $B_E|_{W,\beta}$.
This suffices to show that the monomial terms of $B_E|_{W,\beta}$ are products of $\ell$ distinct variables.
Hence there are constants $c_F$ such that
\[
\det B_E|_{W,\beta} = \sum_{
	\substack{F\subset E \\ \abs{F}=\dim W} } 
	c_F \prod_{f\in F} A_f.
\]
Fix a subset $F$ and compute $c_F$ as follows.
Evaluate both sides at $A_e=1$ for $e\in F$ and $A_e=0$ for $e\notin F$ to determine the constants.
Evaluating the right side gives $c_F$.
Evaluating $\det B_E|_{W,\beta}$ gives the determinant of a bilinear form:
\[
\det \sum_{f\in F} X_f^2|_W .
\]
That is, $c_F = \det \sum_{f\in F} X_f^2|_W$.
With respect to the basis $\beta =\{w_1,\dots,w_\ell\}$, the bilinear form $X_f^2|_W$
is the matrix with $(i,j)$-th element $X_f(w_i)X_f(w_j)$. 
Thus,
\[
\sum_{f\in F} X_f^2|_W = \left(
\sum_{f\in F} X_f(w_i)X_f(w_j)
\right)= \left( X_{f_k}(w_i) \right) \left( X_{f_k}(w_j)\right).
\]
The matrix $(X_{f_k}(w_j))$ is the $\ell\times \ell$ matrix representing the map 
\[
\pi_F|_W:W\to K^F
\] 
projecting $W$
onto its $F$ coordinates (so $k$ is the row index,
and $j$ is the column index), and $(X_{f_k}(w_i))$ is the transpose. Taking determinants yields
\[
\det \sum_{f\in F} X_f^2|_W = \left(\det \pi_F|_W\right)^2 = \Plucker_F(W)^2.
\]
\end{proof}

\subsection{Graph Polynomials}

We use the term \emph{graph} to mean an undirected multigraph.
The set of edges of a graph $G$ will be denoted $E(G)$, and the set of vertices $V(G)$.
If $G$ is clear from the context, we abbreviate this to $E$ and $V$.

For every graph $G$, the \emph{first} and \emph{second graph polynomials} 
appear in parametric Feynman integrals~\cite{IZ},~\cite{BW}. 
Each edge $e$ in $E$ corresponds to a variable $A_e$.
The graph polynomials are polynomials in these variables.
The first graph polynomial is defined as
\[
\Psi_G(A) = 
\sum_{
\substack{
\text{spanning forests}\\
F\subset E}
}
\,\,  \prod_{e\notin F} A_e.
\]
Our definition of spanning forest differs from some references
where a spanning forest is a forest containing all vertices.
The definition is the necessary generalization of spanning tree
to disconnected graphs.
\begin{definition}\label{D:spanningForest}
A \emph{spanning forest} is a forest whose components are spanning trees for
the components of the graph. 
\end{definition}

Recall that the homology of a graph $G$ 
can be computed  as the kernel and cokernel of a \emph{boundary map}
\[
\partial: \mathbb{Z}^E \to \mathbb{Z}^V
\]
defined as follows.
For each edge $e\in E$, the gluing map from $e$ to the vertices of $G$ is defined by mapping the boundary of $e$ to two elements
of $V$ (not necessarily distinct). Denote these elements by $h(e)$ and $t(e)$.
Fix the choice of $h(e)$ and $t(e)$ for each edge $e$.
Let $\mathbb{Z}^E$ and $\mathbb{Z}^V$ be the free abelian
groups on the sets $E$ and $V$. Define the group homomorphism
\[
\partial: \mathbb{Z}^E\to \mathbb{Z}^V
\]
by $\partial(e)=h(e)-t(e)$ and extend linearly.
The homology of a graph $G$ is the homology of $\partial:\mathbb{Z}^E\to \mathbb{Z}^V$. That is, the first homology of $G$ is $H_1(G,\mathbb{Z})=\Ker \partial$,
and the zeroth homology of $G$ is $H_0(G,\mathbb{Z})=\Coker \partial$. 
The definition depends on the \emph{orientation} of the edges (i.e., the choice of $h(e)$ and $t(e)$),
but the homology groups are well-defined up to isomorphism.
To get a configuration, tensor with a field $K$:
\[
0\to H_1(G,\mathbb{Z})\otimes K \to K^E \xrightarrow{\partial\otimes \text{id}} K^V \to H_0(G,\mathbb{Z})\otimes K \to 0.
\]
By direct calculation or the universal coefficient theorem, \[H_i(G,K)=H_i(G,\mathbb{Z})\otimes K.\] 
We 
write $H_i(G)$ for $H_i(G,K)$ and $h_i(G)$ for $\dim_K H_i(G)$. 

\begin{proposition}[\cite{BEK}, Proposition~2.2]
The first graph polynomial $\Psi_{G}$ is a configuration polynomial
for the configuration $H_1(G)\subseteq K^E$.
It can be computed as the configuration determinant in an
integral basis for $H_1(G)$, i.e., a basis for $H_1(G,\mathbb{Z})$.
\end{proposition}

In this section, we prove the corresponding result, Proposition~\ref{P:secondGraphPoly},
for
the second graph polynomial: it is a configuration polynomial
for a relative homology group of $G$. 
Before we can provide the definitions, 
we review some terminology and notation.
\begin{definition}\label{D:quasiSpanningForest}
A \emph{quasi-spanning forest} is a spanning forest of $G$ with one edge
removed.
The complement of a quasi-spanning forest is a
\emph{cut set}.\footnote{The terms in the literature 
that define quasi-spanning forest and cut set are not 
consistent, for example see~\cite{IZ},~\cite{Nakanishi},~\cite{BK},~\cite{BK10}.
In some cases, cut sets are called minimal cut sets to indicate that
removing more edges would create more components. In other cases,
cut sets are simply required to disconnect the graph, and the
complement does not even need to be a forest.
Similarly, quasi-spanning forests are often defined only for
connected graphs, in which case they are often named two-trees or
spanning two-trees.
We believe the use of two in this term refers to the number of components,
which will not be two in the generality of disconnected graphs.}
\end{definition}
We consider the graph as the pair of its set of vertices and its
multiset of edges. 
Therefore, by removing an edge in the definition of quasi-spanning forest,
the set of vertices remains unchanged.
In the topological sense of a graph, only the interior of an edge
is removed.
In particular, removing an edge may leave an isolated vertex.
More generally, if $F$ is a subgraph of $G$, 
we will use the somewhat unconventional notation that $G-F$
is the subgraph of $G$ with the edges of $F$ removed; the set of
vertices $V(G-F)$ is the same as the set of vertices $V(G)$.
This notation reflects the fact that we focus on the
first homology of the graph, which is a subspace
of the space of edges; isolated vertices will not alter this space.

From the definition, a quasi-spanning forest has one more connected 
component than $G$. 
If the spanning forests of $G$ have no edges,
then $G$ has no quasi-spanning forests (e.g., if $G$ has no edges connecting distinct vertices); in the following discussion,
we tacitly assume that the spanning forests of $G$
have edges.

\begin{definition}\label{D:momentum}
A \emph{momentum} of $G$ is an element of
\[
K^{V,0} = \left\{ p\in K^{V} \, \Bigg|\, \sum_{v\in V(G_\alpha)} p_v = 0 \text{ for all connected components }G_\alpha\text{ of }G\right\}.
\]
\end{definition}
In other words, $p$ is an assignment of elements of the field $K$ 
to the vertices of $G$ that   sum to zero
over each connected component of $G$ (i.e., momentum is conserved on each component).
The vertices of $G$ agree with the vertices of any of its quasi-spanning forests $F$, 
so $K^{V(G),0}\subseteq K^{V(F)}$.
Each component of $G\cap F$ is a spanning tree for a component of $G$ with
two exceptions, and therefore, $p$ may not be a momentum for $F$.
If the two components of $F$ that are not spanning trees of
a component of $G$ are $T_1$ and $T_2$,
then using the defining property of the momentum
\[
m_{T_1}(p) = \sum_{ v\in T_1} p_v = - \sum_{ v\in T_2 } p_v = -m_{T_2}(p).
\]
\begin{definition}\label{D:qsfMomentum}
The \emph{momentum associated to the quasi-spanning forest $F$} (or
to the corresponding cut set $C=G-F$) is the following function
from the momentum $K^{V(G),0}$ to $K$:
\[
s_F(p) = s_C(p) = m_{T_1}(p)^2 = m_{T_2}(p)^2.
\]
\end{definition}
The second graph polynomial is
\begin{align*}
\Phi_G(p,A)
&=\sum_{
\substack{\text{quasi-spanning forests}\\ F\subset E}} 
s_{F}(p) \prod_{f\notin F} A_f\\
&=\sum_{
\substack{\text{cut sets}\\ C\subset E}} 
s_{C}(p) \prod_{f\in C} A_f
\end{align*}

Note that $\partial(K^E)\subseteq K^{V,0}$ from the definition of $\partial$, and
the inclusion is an equality by counting dimensions in the homology sequence.
Therefore, for every $p\in K^{V,0}$, we can find $q\in K^E$ such that 
$\partial q=p$. 
\begin{definition}\label{D:relativeHomology}
For a nonzero momentum $p$, 
define the (first) \emph{homology of $G$ relative to $p$} to be
$H_1(G,p) = \partial^{-1}(Kp)=H_1(G)\oplus Kq$.
That is, it is the relative homology group for $Kp$ considered as a chain complex concentrated in degree zero
included in the chain complex for the graph $G$.
\end{definition}
More specifically, the kernel and homology of that inclusion of chain complexes give
\begin{diagram}
H_1(p)=0 & \rTo &  0 & \rTo & Kp & \rTo & H_0(p)=Kp \\ 
   \dTo &             &  \dTo &    & \dInto &   &\dTo \\
 H_1(G) & \rTo & K^E & \rTo^\partial & K^{V} & \rTo & H_0(G) \\ 
 \dInto &             &  \dTo^{=} &    & \dOnto &   & \dTo\\
 H_1(G,p) & \rTo& K^E& \rTo^{\bar{\partial}} & K^{V}/Kp & \rTo & H_0(G,p).\\
\end{diagram}
There is an induced long exact sequence from the snake lemma:
\[\
0\to H_1(G)\to H_1(G,p)\to H_0(p) \to H_0(G)\to H_0(G,p)\to 0.
\]
Because $p\in K^{V,0}$,
\[
H_0(G)\cong K^V/K^{V,0} \cong  K^V/(K^{V,0}+Kp) \cong H_0(G,p),
\]
and the map $H_1(G,p)\to H_0(p)=Kp$ is surjective.
\begin{proposition}\label{P:secondGraphPoly}
The second graph polynomial $\Phi_G(p,A)$ is a configuration polynomial
for the configuration $H_1(G,p)\subseteq K^E$.
It can be computed as the configuration determinant in a basis for
$H_1(G,p)$ consisting of an integral basis for $H_1(G)$
and an element $q\in K^E$ such that $\partial q=p$.
\end{proposition}
\begin{remark}\label{R:orientation}
The definition of the homology groups depends on an arbitrarily chosen orientation
of the graph. 
While this choice is insignificant for computing the homology group
up to isomorphism,
the homology group as a configuration in $K^E$ depends crucially on the choice.
Certainly, different orientations often lead to different configurations.
They all define the same configuration polynomial for the following reason.

Consider two orientations of $G$ that differ for one edge, say $e_1$.
Suppose that $\partial$ and $\tilde{\partial}$ are the boundary maps corresponding
to these two orientations.
The isomorphism
\begin{align*}
\varphi:K^E  &\to K^E\\
x=(x_1,\dots,x_n) & \mapsto (-x_1,x_2,\dots,x_n)
\end{align*}
satisfies $\partial \circ \varphi = \tilde{\partial}$,
providing an isomorphism between the homology groups defined
by the two boundary maps.
Moreover, the bilinear form is invariant:
\[
B_E(A)(x) = \sum_{i=1}^n A_i x_i^2 = B_E(A)\left( \varphi(x) \right).
\]
For orientations that differ at more edges, 
$B_E(A)$ is invariant under the isomorphism defined as a composition
of similarly defined $\{\varphi_j\}_{j\in J}$ for the set of edges $J$ at which the orientation differs.
\end{remark}
The proof of Proposition~\ref{P:secondGraphPoly} is an exercise in the linear algebra of based vector spaces,
which we break up into a few lemmas.
\begin{lemma}[Lifting Momenta from Vertices to Edges] \label{L:momenta}
Let $F$ be a quasi-spanning forest of a graph $G$, let $C$ be its complementary cut set, and let $T_1$ and $T_2$ be the two maximal subtrees of $F$
that are not spanning trees of $G$. 
Let $p$ be a nonzero element of $K^{V,0}$,
and let $q\in K^E$ map to $p$: $\partial q=p$.
Then
\begin{equation}
m_{T_i}(p)=\sum_{\substack{ e\in C \\ h(e)\in T_i}} q_e - \sum_{\substack{ e\in C\\ t(e)\in T_i}} q_e
\label{eq:edgeMomenta}
\end{equation}
for both $i=1,2$. In particular, the momentum associated to $F$ is
\[
s_F(p)=\left(\sum_{\substack{ e\in C \\ h(e)\in T_i}} q_e - \sum_{\substack{ e\in C\\ t(e)\in T_i}} q_e\right)^2.
\]
\end{lemma}
\begin{proof}
The proof is the same for $T_1$ and $T_2$; we prove it for $T_1$.
The basis element $v\in K^V$ has a dual element $X_v\in (K^{V})\spcheck$. 
Denote the restriction of $X_v$ to $K^{V,0}$ by $X_v$ as well.
By definition of $q$ and $m_{T_1}(p)$,
\begin{align*}
m_{T_1}(p)&=\sum_{v\in V(T_1)} p_v\\
&= \sum_{v\in V(T_1)} X_v(p)\\
&= \sum_{v\in V(T_1)} X_v(\partial q)\\
&= \sum_{v\in V(T_1)} X_v\left( \sum_{e\in E(G)} q_e h(e)- q_e t(e) \right)\\
&= \sum_{v\in V(T_1)} \sum_{e\in E(G)} q_e X_v(h(e)- t(e)).
\end{align*}
Let $G_\alpha$ be the component of $G$ containing both $T_1$ and $T_2$.
The edges in $T_2$ and $G_\beta$ for $\beta\neq\alpha$ have vertices disjoint from $V(T_1)$, so
we may remove those edges from the sum:
\begin{equation}\label{eq:momenta}
m_{T_1}(p)= \sum_{v\in V(T_1)}  \sum_{e\in E(G_\alpha-T_2)} q_e X_v(h(e)- t(e) ).
\end{equation}

Note that
\[
X_{v}(h(e)-t(e)) = \begin{cases}
1 & \text{ if }v=h(e), v\neq t(e),\\
-1 &  \text{ if }v=t(e), v\neq h(e),\\
0 & \text{ otherwise.}
\end{cases}
\]
If both $h(e)$ and $t(e)$ are in $V(T_1)$, then 
 $q_e$ appears in the sum~\eqref{eq:momenta} with both a plus and a minus, so it cancels out.
In particular, every edge $e\in T_1$ has both $h(e)$ and $t(e)$ in $V(T_1)$, so
the edges in $T_1$ may be removed from the sum:
\begin{align*}
m_{T_1}(p)&=\sum_{v\in V(T_1)} \sum_{e\in E(G_\alpha-T_1-T_2)} q_e X_v(h(e)-  t(e) ) \\
& = \sum_{v\in V(T_1)} \sum_{e\in C\cap G_\alpha} q_e X_v(h(e)-t(e)).
\end{align*}
Finally, $X_v|_{G-G_\alpha}= 0$ for $v\in V(T_1)$, so we may put the other edges of $C$ back into the formula:
\begin{align*}
m_{T_1}(p)
& = \sum_{v\in V(T_1)} \sum_{e\in C} q_e X_v(h(e)-t(e))\\
& = \sum_{e\in C} \sum_{v\in V(T_1)} q_e X_v(h(e)-t(e))\\
&= \sum_{e\in C} \left( \sum_{h(e)\in V(T_1)} q_e -\sum_{t(e)\in V(T_1)} q_e\right).
\end{align*}
\end{proof}
Let $F$ be a subgraph of a graph $G$.  
The following diagram commutes and its row and column are exact:
\begin{diagram}
	&	&		&		& 0 & 	\\
	&  &  &&  \dTo & & &  &  & \\
	 &  &  &&  \mathbb{Z}^{E(F)} & & &  &  & \\
	 &  &  & &\dTo &  \rdTo^{\alpha}& &  &  & \\
0 & \rTo & H_1(G,\mathbb{Z}) & \rTo &\mathbb{Z}^{E(G)} & \rTo^{\partial}& \mathbb{Z}^{V(G),0} & \rTo &0 \\
  &      &        & \rdTo^{\beta} & \dTo       &     &                   &    & \\
  &      &        &       & \mathbb{Z}^{E(G-F)}& &                  &    &  \\
  &      &        &       & \dTo       &       &               &    &\\
  &      &        &       & 0.          &       &               &     &   \\
\end{diagram}
\begin{lemma}[\cite{BEK}, Lemma~2.1]\label{L:bekSpanningTree}
The subgraph  $F$ is
a spanning forest of $G$ if and only if $\alpha$ is an isomorphism,
which is equivalent to $\beta$ being an isomorphism.
In particular, $\det(\beta )=\pm 1$ when $F$ is a spanning forest and is zero otherwise.
\end{lemma}
\begin{lemma}[Momenta in terms of Pl\"{u}cker Coordinates]\label{L:pluckerMomenta}
Let $G$, $F$, $C$, $T_1$, $T_2$, and $p$ be defined as in Lemma~\ref{L:momenta}.
Let $\beta$ be the projection from $H_1(G,p)$ to $K^C$.
Then in a basis of $H_1(G,p)$ consisting of $q$ in the preimage of $p$ and an integral basis of $H_1(G)$,
\[
\det(\beta)=\pm m_{T_1}(p)=\mp m_{T_2}(p).
\]
In particular, $\Plucker_C(H_1(G,p))^2= s_C(p)$ in this basis.
\end{lemma}
\begin{proof}
Let $m=h_1(G)$.
Write $\beta$ in a basis as in the statement of the lemma: 
\[
\begin{pmatrix}
X_1(q) & X_1(\ell_2) &  \dots & X_1(\ell_{m+1}) \\
X_2(q) & X_2(\ell_2) &  \dots & X_2(\ell_{m+1}) \\
\vdots & \vdots & \ddots & \vdots \\
X_{m+1}(q) & X_{m+1}(\ell_2) &\dots & X_{m+1}(\ell_{m+1})
\end{pmatrix}.
\]
The columns correspond to a $q$ such that $\partial q=p$ followed by the integral basis $\ell_2,\dots,\ell_{m+1}$ of $H_1(G)$. 
The rows correspond to the elements of  the cut set
\[
C=\{e_1,\dots,e_{m+1}\}.
\]
Write $q_i$ for $X_i(q)$.
Expand the determinant of $\beta$ along the first column:
\begin{align}
\det \beta = q_1 \Plucker_{C-e_1}(H_1(G)) -&q_2 \Plucker_{C-e_2}(H_1(G))+\notag\\
&\dots\pm q_{m+1} \Plucker_{C-e_{m+1}}(H_1(G)).\label{eq:detbeta}
\end{align}
The Pl\"{u}cker coordinate $\Plucker_{C-e_i}(H_1(G))$ is computed in an integral basis, so according to Lemma~\ref{L:bekSpanningTree}, 
\[
\Plucker_{C-e_i}(H_1(G)) = 
\begin{cases}
\pm 1 & \text{ if } (G-C)\cup e_i \text{ is a spanning forest,}\\
0 & \text{ otherwise.}
\end{cases}
\]
In particular, the coefficient of $q_i$ in $\det \beta$ vanishes if $(G-C)\cup e_i$ is not a spanning forest, which happens exactly when both
  the head and the tail of $e_i$ are in the same component of $F$, in agreement with Equation~\eqref{eq:edgeMomenta} where $q_e$
  may appear twice in the sum but with opposite signs.
  
We must verify that the nonzero coefficients of the $q_i$ in Equation~\eqref{eq:detbeta}
agree with Equation~\eqref{eq:edgeMomenta}. In this case, we know that $(G-C)\cup e_i$ is a spanning tree, the coefficients
are $\pm 1$,
and we need to check that  
  \begin{enumerate}
  \item the signs on $q_i$
  and $q_j$ in $\det \beta$ agree if the heads of both edges $e_i$ and $e_j$ or the tails of both are in $T_1$;
  \item and the signs on $q_i$ and $q_j$ are opposite if the head of one of the edges $e_i$ or $e_j$ and the tail of the other edge is in $T_1$.
  \end{enumerate}
  Without loss of generality, assume $i=1$ and $j=2$.
  In either case above,
  $(G-C)\cup e_1$ is a spanning forest.
  At this point, the value of $\det \beta$ only depends on the integral basis of $H_1(G)$ for an overall sign, so 
  we are free to choose the basis $\{\ell_2,\dots,\ell_{m+1}\}$ of $H_1(G,\mathbb{Z})$.
  In particular, choose the basis to be a circuit basis with respect to the spanning
  forest $(G-C)\cup e_1$. Recall that a \emph{circuit basis with respect to a spanning forest $F$} is a basis 
  of $H_1(G,\mathbb{Z})$ in which each element is a sum of edges in $F$ plus one edge not in $F$ with coefficients $\pm 1$.
  A circuit basis exists by Lemma~\ref{L:bekSpanningTree}.
  For such a basis and $i$,$j\in \{2,\dots,m+1\}$, the edges $e_i$ are not in $\ell_j$ unless $j=i$.
  For this basis, note that
  \begin{align*}
  X_i(\ell_i)\neq 0 & \text{ for }i\in\{2,\dots, m+1\}\\
  X_i(\ell_j)=0 & \text{ for }j\neq i,\, i,j\in\{2,\dots,m+1\}.
  \end{align*}
Then the determinant begins
\[
\det \beta = q_1 \det \gamma_1 -q_2 \det \gamma_2 + \cdots
\]
where
\begin{align*}
\gamma_1 &= 
\begin{pmatrix}
X_2(\ell_2) & X_2(\ell_3) & \dots & X_2(\ell_{m+1}) \\
X_3(\ell_2) & X_3(\ell_3) & \dots & X_3(\ell_{m+1}) \\
\vdots & \vdots & \ddots & \vdots \\
X_{m+1}(\ell_2) & X_{m+1}(\ell_3) & \dots & X_{m+1}(\ell_{m+1})
\end{pmatrix}\\
&=
\begin{pmatrix}
X_2(\ell_2) & X_2(\ell_3) & \dots & X_2(\ell_{m+1}) \\
0 & X_3(\ell_3) & \dots & X_3(\ell_{m+1}) \\
\vdots & \vdots & \ddots & \vdots \\
0 & X_{m+1}(\ell_3) & \dots & X_{m+1}(\ell_{m+1})
\end{pmatrix}\\
\end{align*}
and
\begin{align*}
\gamma_2&=
\begin{pmatrix}
X_1(\ell_2) & X_1(\ell_3) & \dots & X_1(\ell_{m+1}) \\
X_3(\ell_2) & X_3(\ell_3) & \dots & X_3(\ell_{m+1}) \\
\vdots & \vdots & \ddots & \vdots \\
X_{m+1}(\ell_2) & X_{m+1}(\ell_3) & \dots & X_{m+1}(\ell_{m+1})
\end{pmatrix}\\
&=
\begin{pmatrix}
X_1(\ell_2) & X_1(\ell_3) & \dots & X_1(\ell_{m+1}) \\
0 & X_3(\ell_3) & \dots & X_3(\ell_{m+1}) \\
\vdots & \vdots & \ddots & \vdots \\
0 & X_{m+1}(\ell_3) & \dots & X_{m+1}(\ell_{m+1})
\end{pmatrix}.
\end{align*}
In fact, many other entries are zero by our choice of basis,
but we have only expressed the zeros where we need them.
With the exception of the first row, $\gamma_1$ and $\gamma_2$ agree, so by the multilinearity of the determinant,
\begin{align*}
\det(\gamma_1)&\pm \det(\gamma_2) =\\ & 
\det 
\begin{pmatrix}
(X_2\pm X_1)(\ell_2) & (X_2 \pm X_1)(\ell_3) & \dots & (X_2\pm X_1)(\ell_{m+1}) \\
0 & X_3(\ell_3) & \dots & X_3(\ell_{m+1}) \\
\vdots & \vdots & \ddots & \vdots \\
0 & X_{m+1}(\ell_3) & \dots & X_{m+1}(\ell_{m+1})
\end{pmatrix}.
\end{align*}
  
Because we have assumed that $e_1$ and $e_2$ fit into one of the two cases above, 
both $(G-C)\cup e_1$ and $(G-C)\cup e_2$ are spanning forests
and have no first homology. The circuit generator $\ell_2$
of $H_1( (G-C)\cup e_1\cup e_2)$ must contain both $e_1$ and $e_2$ because 
\begin{align*}
h_1(\ell_2-e_1)&\leq h_1( (G-C)\cup e_2) =0 \\
h_1(\ell_2-e_2)&\leq h_1( (G-C)\cup e_1) =0.
\end{align*}
Therefore,
\begin{align*}
X_1(\ell_2)=X_2(\ell_2) =\pm 1 &\text{ if only one of }e_1\text{ or }e_2\text{ has its head in }T_1\text{, and}\\
X_1(\ell_2)=\pm X_2(\ell_2) =\pm 1 &\text{ if  both }e_1\text{ and }e_2\text{ have their heads in the same }T_i.
\end{align*}
By changing the orientation of every edge in $\ell_2$ if necessary, we may assume $X_2(\ell_2)=1$ so that the cases
are $X_1(\ell_2)=1$ or $X_1(\ell_2)=-1$.
The two cases noted above become
\begin{align*}
(X_2-X_1)(\ell_2) &= 0 \text{ if } X_1(\ell_2)=1\\
(X_2+X_1)(\ell_2) &= 0 \text{ if } X_1(\ell_2)=-1.
\end{align*}
Matching signs, the first column of the preceding matrix is zero.
Therefore,
\begin{align*}
\det \gamma_1 = \begin{cases}
\det \gamma_2 &\text{ if only one of $e_1$ or $e_2$ has its head in $T_1$,}\\
 -\det \gamma_2 &\text{ if both $e_j$ have their heads in the same $T_i$,}
 \end{cases}
\end{align*}
and
\begin{align*}
\det \beta = \begin{cases}
(\det\gamma_1) (q_1-q_2)+\dots &\text{ if only one of }h(e_1)\text{ or }h(e_2)\text{ is in }T_1,\\
 (\det\gamma_1) (q_1+q_2)+\dots &\text{ if both }h(e_j)\in T_i.
 \end{cases}
\end{align*}
This completes the comparison of the signs on $q_1$ and $q_2$ and the proof.
\end{proof}

\begin{proof}[Proof of Proposition~\ref{P:secondGraphPoly}]
By Proposition~\ref{P:configurationPolynomial}, the configuration determinant is
\[
\Psi_{H_1(G,p)}(A)=\sum_{
	\substack{C\subset E(G) \\ \abs{C}=h_1(G)+1 } } 
	\Plucker_C(H_1(G,p))^2 \prod_{e\in C} A_e.
\]
By Lemma~\ref{L:pluckerMomenta},  $\Plucker_C(H_1(G,p))^2= s_C(p)$ when $C$ is a cut set.

The final step is to show that if $\Plucker_C(H_1(G,p))^2$ is nonzero,
then $C$ is a cut set.
There is a commutative diagram exact in the row and column
\begin{diagram}
	&	&		&		& 0 & 	\\
	&  &  &&  \dTo & & &  &  & \\
	 &  &  &&  K^{E(G-C)} & & &  &  & \\
	 &  &  & &\dTo &  \rdTo & &  &  & \\
0 & \rTo & H_1(G,p) & \rTo &\mathbb{K}^{E(G)} & \rTo^{\partial}& K^{V,0}/Kp & \rTo &0 \\
  &      &        & \rdTo^{\pi_C} & \dTo       &     &                   &    & \\
  &      &        &       & K^{C}& &                  &    &  \\
  &      &        &       & \dTo       &       &               &    &\\
  &      &        &       & 0.          &       &               &     &   \\
\end{diagram}
Note that $\Plucker_C(H_1(G,p))$ is the determinant of $\pi_C$ is some basis, 
so
if $\Plucker_C(H_1(G,p))$ is nonzero, then $\pi_C$
is an isomorphism.
For every edge $e\in C$, the natural map $\beta_e: H_1(G)\to K^{C-e}$ factors as 
\[
\beta_e:  H_1(G) \hookrightarrow H_1(G,p) \xrightarrow[\cong]{\pi_C} K^C \twoheadrightarrow K^{C-e}
\]
The image of $H_1(G)$ in $K^C$ has codimension one, so there is 
at least one $\tilde{e}\in C$ 
for which projecting to $K^{C-\tilde{e}}$ makes $\beta_{\tilde{e}}$ an isomorphism.
By Lemma~\ref{L:bekSpanningTree}, $(G-C)\cup \tilde{e}$ is a spanning forest, which means
that $C$ is a cut set.
\end{proof}
\begin{remark}
One could define higher-order graph polynomials by replacing $Kp$ with higher dimensional
subspaces of $K^{V,0}$ or even replacing $0\to Kp$ with other subcomplexes of $K^E\to K^{V,0}$.
To our knowledge, it is an open problem to prove combinatorial formulas
for these generalizations of graph polynomials.
Based on the review~\cite{BW}, we suspect such higher-order graph polynomials
may correspond to the polynomials $\mathcal{W}^{(j)}$ defined there.
Moreover, a general framework for such a correspondence might simply
need to ``lift'' the all-minor matrix-tree theorem by using the differential $\partial:\mathbb{Z}^E\to\mathbb{Z}^V$.
\end{remark}

\section{Singularities of Configuration Hypersurfaces}\label{S:hypersurfaces}
In this section, we prove the main result, Theorem~\ref{T:singularityRank}.
The differentiation in this section requires
us to assume that we are working over a field of characteristic zero.
First, we describe how differentiating configuration polynomials 
yields other configuration polynomials in Section~\ref{S:restrictionsAndProjections}.
The main result is proven in Section~\ref{S:singularities}.
An application to tangent cones 
is provided in Section~\ref{S:tangentSpaces}
where we require the field to be algebraically closed.
\subsection{Restrictions of Configurations} \label{S:restrictionsAndProjections}
Every based vector space $K^E$ has natural based subspaces $K^F$ for $F\subseteq E$. 
%
For a configuration $W$, there is an induced configuration in $K^F$: the \emph{restriction} $W^F= K^F\cap W$. 
%
Every configuration is a pair of a based vector space and a subspace, so $\left(W^F,K^F\right)$ and $\left(W^F,K^E\right)$
are two different configurations.
For example, if $W\subseteq K^F\subseteq K^E$ for some subset $F\subseteq E$, then  the formulas for the
configuration polynomials given by Proposition~\ref{P:configurationPolynomial} are
\[
\Psi_{\left(W,K^F\right)}(A)=
\sum_{
	\substack{G\subset F \\ \abs{G}=\dim W} } 
	\Plucker_G(W)^2 \prod_{g\in G} A_g,
\]
and
\[
\Psi_{\left(W,K^E\right)}(A)=
\sum_{
	\substack{G\subset E \\ \abs{G}=\dim W} } 
	\Plucker_G(W)^2 \prod_{g\in G} A_g.
\]
The Pl\"{u}cker coordinates 
of $\mathbb{P}(\bigwedge^\ell K^F)$ in $\mathbb{P}(\bigwedge^\ell K^E)$
vanish for all subsets $G\subset E$  that are not contained in $F$.
Therefore, these polynomials are identical.
We suppress the based vector space from the notation unless the context demands it.
\begin{example} \label{E:inducedconfigurationPolynomial1}
We will continue with the configuration $W$ of Example~\ref{E:configuration1}, whose configuration
polynomial we computed to be
\[
\Psi_W(A)=A_1A_2+4A_1A_3+4A_2A_3.
\]

Let $F=\{e_2,e_3\}\subseteq E$.
%
Consider the restriction configuration $W^F$.
In the basis $2e_3-e_2$ for $W^F$, the bilinear form $B_F = A_2 X_2^2+A_3 X_3^2$ restricted from $K^F$ is
\[
B_F|_{W_F}=
A_2
\begin{pmatrix}
1
\end{pmatrix}
+
A_3
\begin{pmatrix}
4
\end{pmatrix}
= \begin{pmatrix}
A_2+4A_3
\end{pmatrix}.
\]
The corresponding configuration polynomial is $\Psi_{W^F}(A)=A_2+4A_3$.
In the same basis for $W^F$, the bilinear form $B_E = A_1X_1^2+A_2X_2^2+A_3X_3^2$ restricted from $K^E$ is
\[
B_E|_{W_F}=
A_1
\begin{pmatrix}
0
\end{pmatrix}
+
A_2
\begin{pmatrix}
1
\end{pmatrix}
+
A_3
\begin{pmatrix}
4
\end{pmatrix}
= \begin{pmatrix}
A_2+4A_3
\end{pmatrix}.
\]
Thus, $\Psi_{\left(W^F,K^E\right)}(A_1,A_2,A_3)=\Psi_{W^F}(A_2,A_3)$.
In fact,
\[
\frac{\partial}{\partial A_1} \Psi_W(A_1,A_2,A_3) = \Psi_{W^F}(A_2,A_3).
\]
\end{example}

Restrictions play a  prominent role
in analyzing the singularities of configuration hypersurfaces.
We start with the case of $W^{E-e}$ for simplicity, and then proceed by induction.
By Proposition~\ref{P:configurationPolynomial}, the configuration polynomial for 
a nonzero restriction $W^{E-e}$ is
\[
\Psi_{W^{E-e}}(A)=
\sum_{
	\substack{F\subset E-e \\ \abs{F}=\dim W^{E-e}} } 
	\Plucker_F\left(W^{E-e}\right)^2 \prod_{f\in F} A_f.
\]
For comparison, the partial of $\Psi_W(A)$ with respect to $A_e$ is
\[
\partial_e \Psi_W(A)= \sum_{ \substack{F\subset E \\ \abs{F}=\dim W \\ e\in F} }
\Plucker_F(W)^2 \prod_{f\in F-e} A_f.
\]
Here we write $\partial_e$ for $\partial_{A_e}$.
Using the bijection $F\mapsto F-e$ between subsets of $E$ of size $\dim W$ containing $e$
and subsets $\tilde{F}$ of $E-e$ of size $\dim W-1$, 
\begin{equation} \label{eq:onepartial}
\partial_e \Psi_W(A)  
=\sum_{\substack{\tilde{F}\subset E-e \\ \abs{\tilde{F}}=\dim W-1}} \Plucker_{\tilde{F}\cup e}(W)^2 
\prod_{f\in \tilde{F}} A_f.
\end{equation}
This last polynomial differs from $\Psi_{W^{E-e}}$ for two possible reasons:
\begin{enumerate}
\item
in general,
\[
\Plucker_{\tilde{F}\cup e}(W)^2\neq \Plucker_{\tilde{F}}\left(W^{E-e}\right)^2;
\]
\item 
if $W\subseteq K^{E-e}$, then $W=W^{E-e}$, and the sets $\tilde{F}$ indexing the sum
in $\partial_e \Psi_W(A)$ have size $\dim W-1$,
but the sets $F$ indexing the sum in $\Psi_{W^F}(A)$ have size $\dim W$.
\end{enumerate}
The following lemmas clarify these differences,
which culminate in Corollary~\ref{C:partialsOfConfigurationPolynomials}.
\begin{lemma}[Trivial Restriction, Simple Case] \label{L:trivialPartialOne}
Suppose that $W$ is a nonzero configuration in $K^E$. The following conditions are equivalent.
\begin{enumerate}
\item $W\subseteq K^{E-e}$,
\item $W=W^{E-e}$,
\item $\Psi_W(A) = \Psi_{W^{E-e}}(A)$,
\item $\Plucker_F(W) = \Plucker_F\left(W^{E-e}\right)$ for all $F\subseteq E$ with
$\abs{F}=\dim W$,
\item $\Plucker_F(W) = 0$ if $e\in F$, $\abs{F}=\dim W$,
\item $A_e$ does not appear in the polynomial $\Psi_W(A)$,
\item $\partial_e \Psi_W(A) = 0$.
\end{enumerate}
\end{lemma}
\begin{proof}
The following implications are straightforward from 
the definitions and Proposition~\ref{P:configurationPolynomial}:
\[
(1) \Leftrightarrow (2) \Rightarrow (3) \Rightarrow (4) \Rightarrow (5) \Leftrightarrow (6) \Leftrightarrow (7).
\]
To complete the equivalence, we show that (5) implies (1). Note
that $\mathbb{P}\left(\bigwedge^\ell K^{E-e}\right)$ is the linear subspace of $\mathbb{P}\left(\bigwedge^\ell K^E\right)$
defined in the Pl\"{u}cker coordinates by $\Plucker_F=0$ for those $F$ containing $e$.
In particular, if $\Plucker_F(W)=0$ for all $F$ containing $e$, then $\det W \in \mathbb{P}\left(\bigwedge^\ell K^{E-e}\right)$.
Therefore, $W\subseteq K^{E-e}$. 
\end{proof}
The previous lemma generalizes by induction to subsets of $E$ with more than one element.
\begin{lemma}[Trivial Restriction] \label{L:trivialPartialMore}
Let $W$ be a nonzero configuration in $K^E$.
The following are equivalent for a subset $H\subseteq E$:
\begin{enumerate}
\item $W\subseteq K^H$,
\item $W=W^H$,
\item $\Psi_{W}(A)=\Psi_{W^H}(A)$,
\item $\Plucker_F(W)=\Plucker_F\left(W^H\right)$ for all $F\subseteq E$ 
with $\abs{F}=\dim W$,
\item $\Plucker_F(W)=0$ if $F\cap (E-H)\neq \emptyset$, $\abs{F}=\dim W$,
\item $A_e$ does not appear in $\Psi_{W}(A)$ for all $e\in E-H$,
\item $\partial_e \Psi_{W}(A) = 0$ for all $e\in E-H$.
\end{enumerate}
\end{lemma}
\begin{proof}
This follows by induction on the size of $E-H$ from Lemma~\ref{L:trivialPartialOne}.
\end{proof}
The case when $W$ is not contained in $ K^{H}$ is described by formally differentiating $\Psi_W(A)$.
Again, we start with $H=E-e$ and proceed by induction. 
First, we state a well-known fact from linear algebra that we use several times:
\begin{lemma}\label{L:intersectionDimensionBound}
Suppose $U_1$ and $U_2$ are subspaces of a vector space $V$. Then
\[
\dim U_1\cap U_2 \geq \dim U_1 -\codim U_2,
\]
which can be summarized as the dimension of a subspace may drop by at most $k$ when intersected with a codimension $k$ subspace.
\end{lemma}
\begin{lemma}[Nontrivial Restriction, Simple Case] \label{L:nontrivialPartialOne}
Suppose $W$ is a configuration in $K^E$, and let $e$ be an element of $E$ for which $W^{E-e}$ is nonzero.
The following conditions are equivalent:
\begin{enumerate}
\item $W\nsubseteq K^{E-e}$,
\item $W^{E-e}$ is a hyperplane in $W$, and there is a nonzero constant $C$ such that for all $F\subset E-e$ with
$\abs{F}=\dim W^{E-e}$ 
\[\Plucker_F\left(W^{E-e}\right) = C \Plucker_{F\cup e}(W),\]
\item  $\langle\partial_e \Psi_{W}(A) \rangle = \langle \Psi_{W^{E-e}}(A) \rangle$
as ideals in the polynomial ring $K[A]$.
\end{enumerate}
\end{lemma}
\begin{proof}
(1) $\Rightarrow$ (2): By Lemma~\ref{L:intersectionDimensionBound} and the assumption that $W\nsubseteq K^{E-e}$, 
the dimension of $W$ drops by $1$ when
intersected with the hyperplane $K^{E-e}$.

The projections 
\[
\pi_{F}|_{W^{E-e}}: W^{E-e} \to K^F
\]
have factorizations
\[
\pi_F|_{W^{E-e}}: W^{E-e} \hookrightarrow W \xrightarrow{\pi_{F\cup e}|_W} K^{F\cup e} \to K^F.
\]
Extend a basis $\{ w_1,\dots, w_{m-1}\}$ of $W^{E-e}$ to a basis for $W$ by adding an element $w_m$.
The matrix representing the map from $W$ to $K^{F\cup e}$ in this basis for $W$ and the
canonical basis of $K^{F\cup e}$ is
\begin{align*}
&\begin{pmatrix}
X_{f_1}(w_1) & X_{f_1}(w_2) & \dots & X_{f_1}(w_m) \\
X_{f_2}(w_1) & X_{f_2}(w_2) & \dots & X_{f_2}(w_m) \\
\vdots & \vdots & \ddots & \vdots \\
X_{e}(w_1) & X_{e}(w_2) & \dots &X_e(w_m)
\end{pmatrix}\\
&\qquad \qquad =
\begin{pmatrix}
X_{f_1}(w_1) & X_{f_1}(w_2) & \dots & X_{f_1}(w_m) \\
X_{f_2}(w_1) & X_{f_2}(w_2) & \dots & X_{f_2}(w_m) \\
\vdots & \vdots & \ddots & \vdots \\
0 & 0 & \dots &X_e(w_m)
\end{pmatrix}.
\end{align*}
Here $f_1,\dots,f_{m-1}$ denote the elements of $F$.
Expanding the above determinant in the last row, 
whose only nonzero element is $X_e(w_m)$, gives
\begin{align*}
 \Plucker_F(W) &=(-1)^{2m} X_e(w_m) \Plucker_F\left(W^{E-e}\right) \\
 & = X_e(w_m) \Plucker_F\left(W^{E-e}\right).
\end{align*}
In particular, $C= X_e(w_m)$ satisfies the conclusion of (2). The choice of $w_m$ can change the overall
constant, but it is independent of the subsets $F$.

(2) $\Rightarrow$ (3): As noted in Equation~\eqref{eq:onepartial},
\[
\partial_e \Psi_W(A) = 
\sum_{\substack{\tilde{F}\subset E-e \\ \abs{\tilde{F}}=\dim W-1}} \Plucker_{\tilde{F}\cup e}(W)^2 
\prod_{f\in \tilde{F}} A_f.
\]
By the assumption (2), the coefficients can be simplified to  
\[
\partial_e \Psi_W(A) = 
C^2 \sum_{\substack{\tilde{F}\subset E-e \\ \abs{\tilde{F}}=\dim W-1}} \Plucker_{\tilde{F}}\left(W^{E-e}\right)^2 
\prod_{f\in \tilde{F}} A_f.
\]
That is,
\[
\partial_e \Psi_W(A) = C^2 \Psi_{W^{E-e}}(A).
\]
The ambiguity of the overall constant is removed by considering the ideals.

(3) $\Rightarrow$ (1): By assumption there is a unit $C$ such that $\partial_e \Psi_W(A) = C \Psi_{W^{E-e}}(A)$. 
By the assumption that $W^{E-e}$ is nonzero, one of its Pl\"{u}cker coordinates is nonzero, so $\partial_e\Psi_W(A) \neq 0$ by Proposition~\ref{P:configurationPolynomial}.
If $W$ were contained in $K^{E-e}$, then $\partial_e\Psi_W(A)= 0$ by Lemma~\ref{L:trivialPartialOne}, which
would be a contradiction.
Thus, $W\nsubseteq K^{E-e}$.
\end{proof}
Again, we generalize to larger subsets of $E$ and higher derivatives.
For  $F=\{e_1,\dots,e_k\}\subseteq E$,
let $\partial_F$ denote $\partial_{e_1}\cdots \partial_{e_k}$.
\begin{lemma}[Nontrivial Restriction] \label{L:nontrivialPartialMore}
Let $W$ be a configuration in $K^E$.
The following conditions are equivalent for a subset
$H\varsubsetneq  E$ assuming $W^H$ is nonzero:
\begin{enumerate}
\item for all sets $H'$ with 
$
H\varsubsetneq H' \subseteq E,
$
there exists $e\in H'-H$ such that
\[
W^{H'}\nsubseteq K^{H'-e},
\]
\item $W^H$ is a codimension $\abs{E-H}$ subspace of $W$, and
there is a nonzero constant $C$ such that for all $F\subset H$ with
$\abs{F}=\dim W^{H}$, 
\[\Plucker_F\left(W^{H}\right) = C \Plucker_{F\cup (E-H)}\left(W\right),
\]
\item $\langle \partial_{(E-H)} \Psi_{W}(A) \rangle = \langle \Psi_{W^{H}}(A)\rangle $.
\end{enumerate}
If these conditions are not satisfied for $H$, then $\Psi_{W^H}(A) = \Psi_{W^{H'}}(A)$
for each $H'$ for which condition (1) fails and $\partial_{(E-H)}\Psi_W(A) = 0$.
\end{lemma}
\begin{proof}
The proof proceeds by induction on the size of $\abs{E-H}$.
At each stage of the induction, first prove that if $H$ does
not satisfy (1), then 
$\Psi_{W^H}(A) = \Psi_{W^{H'}}(A)$
for each $H'$ for which condition (1) fails and \[\partial_{(E-H)}\Psi_W(A) = 0.\]
Then proceed to prove that 
\[
\text{(1)} \Rightarrow \text{(2)} \Rightarrow \text{(3)} \Rightarrow \text{(1)}.
\]
The reason for this complicated induction is twofold. First, the inductive hypothesis
for the equivalence of the three conditions is used to prove what
happens when (1) fails. Second, the situation in which (1) fails is used to prove
(3) $\Rightarrow$ (1) by contradiction.

First consider the case $\abs{E-H}=1$.
If $H$ does not satisfy (1), then Lemma~\ref{L:trivialPartialOne}
implies that $\Psi_{W^H}(A)=\Psi_W(A)$ and $\partial_{(E-H)}\Psi_W(A)= 0$.
The equivalence of the three conditions is
the content of Lemma~\ref{L:nontrivialPartialOne}.
Now consider $H\subseteq  E$ with $\abs{E-H}>1$. 

If $H$ does not satisfy (1), 
there are sets $H'$ such that
\[
H\varsubsetneq H' \subseteq E
\]
and for all elements $e\in H'-H$, 
\[
W^{H'}\subseteq K^{H'-e}.
\]
In particular, $W^{H'} \subseteq K^H$ for such subsets, and thus
$W^H=W^{H'}$ by Lemma~\ref{L:trivialPartialMore}.

Let $\tilde{H}$ be  a maximal such subset.
If $\tilde{H}=E$, then $\partial_{E-H} \Psi_W= 0$ by
Lemma~\ref{L:trivialPartialMore}.
If $\tilde{H} \varsubsetneq E$, then by maximality, 
for every set $H'$ such that $\tilde{H}\varsubsetneq H' \subseteq E$, 
there is an element $e'\in H'-H$ such that $W^{H'} \nsubseteq K^{H'-e'}$.
There must be such an element $e'$ in the smaller set $H'-\tilde{H}$,
otherwise Lemma~\ref{L:trivialPartialMore} implies that $W^{H'}=W^{\tilde{H}}$,
and thus 
\[
W^{H'}= W^{\tilde{H}}=W^{H} \subseteq K^{H'-e'}.
\]
Therefore, the inductive hypothesis for the equivalence of the three conditions applies to $\tilde{H}$, 
so there is a nonzero constant $C$ such that 
\[
\partial_{(E-\tilde{H})}\Psi_W(A) = C\Psi_{W^{\tilde{H}}}(A).
\]
For every $e\in \tilde{H}-H$,
\[
\partial_e \Psi_{W^{\tilde{H}}}(A) = 0
\]
by Lemma~\ref{L:trivialPartialMore}.
Thus, 
\begin{align*}
\partial_{(E-H)}\Psi_W(A) & = \partial_{(\tilde{H}-H)}\partial_{(E-\tilde{H})}\Psi_W(A) \\
													& = C\partial_{(\tilde{H}-H)}\Psi_{W^{\tilde{H}}}(A) \\
													& = 0.
\end{align*}

(1) $\Rightarrow$ (2): 
Assuming the condition in (1) holds for $H$, taking $H'=E$, there is an element $\tilde{e}\in E-H$ such that
$W \nsubseteq K^{E-\tilde{e}}$.
By Lemma~\ref{L:nontrivialPartialOne}, $W^{E-\tilde{e}}$ is a codimension $1$ subspace of $W$
and there is a nonzero constant $\tilde{C}$ such that 
\[
\Plucker_{\tilde{F}}\left( W^{E-\tilde{e}} \right) = \tilde{C} \Plucker_{\tilde{F}\cup\tilde{e}}(W)
\]
for all $\tilde{F}\subset E-\tilde{e}$ with $\abs{\tilde{F}} = \dim W^{E-\tilde{e}}$.

The condition (1) applies to all $H'$ with $H\varsubsetneq H' \subseteq E-\tilde{e}$, 
and 
\[
\abs{E-\tilde{e}-H} < \abs{E-H},
\] 
so by the inductive hypothesis applied to the
configuration $W^{E-\tilde{e}}$ in $K^{E-\tilde{e}}$, 
$W^H$ is a codimension $\abs{E-\tilde{e}-H}$ subspace of $W^{E-\tilde{e}-H}$ and there
is a nonzero constant $C'$ such that
\[
\Plucker_F\left(W^H\right) = C' \Plucker_{F\cup(E-\tilde{e}-H)}\left( W^{E-\tilde{e}} \right)
\]
for all $F\subset H$ with $\abs{F} = \dim W^H$.
Therefore, $W^H$ is a codimension $\abs{E-H}$ subspace of $W$, and there is a nonzero 
constant $C=C'\tilde{C}$ such that
\begin{align*}
\Plucker_F\left(W^H\right) &= C' \Plucker_{F\cup(E-\tilde{e}-H)}\left( W^{E-\tilde{e}} \right)\\
&= C'\tilde{C}\Plucker_{F\cup(E-H)}(W)
\end{align*}
for all $F\subset H$ with $\abs{F} = \dim W^H$.
%

(2) $\Rightarrow$ (3):
In general, using the bijection between subsets of $E$ containing $E-H$ of size $\dim W$ and subsets of $H$ of size $\dim W-\abs{E-H}$,
\begin{align*}
\partial_{(E-H)} \Psi_{W} (A) &= \sum_{\substack{ F\subseteq E\\ \abs{F}=\dim W \\ E-H\subseteq F }} 
\Plucker_F\left(W\right)^2 \prod_{f\in F-(E-H)} A_f\\
&= \sum_{\substack{ \tilde{F}\subseteq H \\ \abs{\tilde{F}}=\dim W-\abs{E-H}  }} 
\Plucker_{\tilde{F}\cup(E-H)}\left(W\right)^2 \prod_{f\in \tilde{F}} A_f.
\end{align*}
By assumption (2), the Pl\"{u}cker coordinates and $\dim W - \abs{E-H}$ become
\begin{align*}
\partial_{(E-H)} \Psi_{W} (A) &= C^2\sum_{\substack{ \tilde{F}\subseteq H \\ \abs{\tilde{F}}=\dim W^H  }} 
\Plucker_{\tilde{F}}\left(W^H\right)^2 \prod_{f\in \tilde{F}} A_f\\
&=C^2 \Psi_{W^H}(A).
\end{align*}
Therefore, the ideals are the same.

(3) $\Rightarrow$ (1):
By assumption, $W^H$ is nonzero, so one of its Pl\"{u}cker coordinates is nonzero, and
 Proposition~\ref{P:configurationPolynomial} shows
that $\Psi_{W^H}(A)$ is not identically zero.

We prove by contradiction, so we will assume (1) fails and show that this implies that (3) fails. 
In fact, when (1) fails, we have already shown that 
\[
\partial_{(E-H)}\Psi_{W}(A)= 0,
\]
which contradicts condition (3) that there is a unit $C$ such that 
\[
\partial_{(E-H)}\Psi_{W}(A) = C\Psi_{W^H}(A)\neq 0.
\]
\end{proof}
\begin{remark}
When Lemmas~\ref{L:nontrivialPartialOne} and~\ref{L:nontrivialPartialMore}
are applied to graph polynomials,
there is no arbitrary scaling of the configuration polynomials 
because an integral basis fixes the coefficients to be $1$ or $0$.
In particular, the arbitrary nonzero constants $C$ may be replaced by $1$,
and the equality of ideals may be replaced by equality of polynomials.

Moreover, if $G_1$ is a subgraph of $G$, then a boundary map for $G$ 
restricts to a boundary map for $G_1$. 
Therefore, $H_1(G_1)$ is the restriction of $H_1(G)$ to $K^{E(G_1)}$.
Whether $\Psi_{G_1}(A)$ can be computed by $\partial_{E(G-G_1)}\Psi_G(A)$
depends on whether $E(G-G_1)$ is contained in the complement of
a spanning forest of $G$.

Similarly for the second graph polynomial, $H_1(G_1,p)$ is the restriction of 
$H_1(G,p)$ to $K^{E(G_1)}$ if $p$ is a momentum on $G_1$.
If $p$ is not a momentum on $G_1$, then the restriction of $H_1(G,p)$
to $K^{E(G_1)}$ is $H_1(G_1)$.
Whether the restriction can be computed by differentiating depends
on whether $E(G-G_1)$ is contained in a cut set.
\end{remark}
The following corollary summarizes the preceding results for
the applications to follow.
\begin{corollary}[Restrictions of Configuration Polynomials] \label{C:partialsOfConfigurationPolynomials}
Let $W$ be a nonzero configuration of dimension $m$ in $K^E$.
For every integer $k$, $1\leq k \leq m-1$, the following ideals in $K[A]$ are the same
\[
\Big\langle \partial_{(E-H)}\Psi_W(A) \, \Big|\, \abs{E-H}\leq k \Big\rangle =
\Big\langle \Psi_{W^H}(A) \, \Big|\, \abs{E-H}\leq k    \Big\rangle.
\]
\end{corollary}
\begin{proof}
First note that the condition  $ k \leq m-1$
implies that for $\abs{E-H}\leq k$
\begin{align}
\dim W^H &= \dim W\cap K^H \notag \\
& \geq m-k \tag{Lemma~\ref{L:intersectionDimensionBound}}\\
& \geq m- (m-1)\notag \\
& =1. \notag
\end{align}
In particular, the configurations $W^H$ are nonzero, 
which is required for citing Lemma~\ref{L:nontrivialPartialOne} below. 

Proceed by induction.
Let
\[
\mathcal{A}_k=\Big\langle \Psi_{W^{H}}(A) \, \Big|\, \abs{E-H}\leq k    \Big\rangle
\]
and
\[
\mathcal{B}_k=\Big\langle \partial_{(E-H)}\Psi_W(A) \, \Big|\, \abs{E-H}\leq k    \Big\rangle.
\]
Consider first the case $k=1$, and suppose $H$ is a subset for which $\abs{E-H}=1$.
If condition (1) of Lemma~\ref{L:nontrivialPartialOne} fails for $H$,
then 
\[ \partial_{(E-H)} \Psi_{W}(A)  = 0 \in \mathcal{A}_1 \]
and
\[ \Psi_{W^{H}}(A) =  \Psi_{W}(A).\]
Note that $\Psi_W(A)$ is an element of $\mathcal{B}_1$ by Euler's formula for homogeneous polynomials
or by defining $\partial_{\emptyset}\Psi_W(A) = \Psi_W(A)$.
If condition (1) is satisfied by $H$, then Lemma~\ref{L:nontrivialPartialOne} 
shows that
\[
\langle \partial_{(E-H)} \Psi_{W}(A) \rangle = \langle \Psi_{W^{H}}(A)\rangle.
\]
Summing these ideals over all $H$ for which $\abs{E-H}\leq 1$ gives $\mathcal{A}_1=\mathcal{B}_1$.

Suppose $\mathcal{A}_{k-1}= \mathcal{B}_{k-1}$, 
and consider an $H$ for which $\abs{E-H} = k$.  
If $H$ satisfies condition (1) of Lemma~\ref{L:nontrivialPartialMore},
then 
$\langle \partial_{(E-H)} \Psi_{W}(A) \rangle = \langle \Psi_{W^{H}}(A)\rangle$ by that lemma.
If $H$ does not satisfy condition (1) of Lemma~\ref{L:nontrivialPartialMore},
then 
\[ \partial_{(E-H)} \Psi_{W}(A)  = 0 \in \mathcal{A}_k, \]
and there is an $H'\varsupsetneq H$ for which
\[ \Psi_{W^{H}}(A) =  \Psi_{W^{H'}}(A) \in \mathcal{B}_{k-1}\subseteq \mathcal{B}_k.\]
Therefore, adding these generators to $\mathcal{A}_{k-1}$ and $\mathcal{B}_{k-1}$ for all $H$ for which $\abs{E-H}=k$ proves
$\mathcal{A}_{k}=\mathcal{B}_{k}$.
\end{proof}
\begin{remark}
If $k\geq m$ in Corollary~\ref{C:partialsOfConfigurationPolynomials}, then  the ideal
\[
\mathcal{A}_k=\Big\langle \Psi_{W^{H}}(A) \, \Big|\, \abs{E-H}\leq k    \Big\rangle
\]
may  not be defined because $W^{H}$ may be zero and we have no convention for the configuration polynomial of the zero configuration.
Note, however, that there may be $H\subseteq E$, $\abs{E-H}\geq m$ for which $W^{H}$ is nonzero.
For example, if $W=H_1(G)$ is a graph configuration, then the restriction $W^{H}$ is
the configuration for the graph with edges $H$ and the same vertices~\cite{BEK}.
If $\abs{E-H}\geq m$, the graph with edges $H$ will have nonzero first homology
as long as $H$ is not a forest.

For $k<m$, the ideals form an ascending chain
\[
\mathcal{A}_1\subseteq \cdots \subseteq \mathcal{A}_{m-1}.
\]
If the convention for the zero configuration were that its configuration polynomial was a unit, then
$\mathcal{A}_m=K[A]$, which completes the chain.
On the other hand, the ideals
\[
\mathcal{B}_k =\Big\langle \partial_{(E-H)}\Psi_W(A) \, \Big|\, \abs{E-H}\leq k \Big\rangle
\]
are defined for all $k$. 
In fact, $\mathcal{B}_k=K[A]$ for $k\geq m$ because $\partial_{(E-H)} \Psi_W(A) = \Plucker_{(E-H)}(W)^2$
if $\abs{E-H}=m$ by Proposition~\ref{P:configurationPolynomial}, 
and at least one of these Pl\"{u}cker coordinates is nonzero.
\end{remark}

\subsection{Singularity-Rank Correspondence}\label{S:singularities}
This section describes the singularities of the configuration hypersurfaces in terms of their rank.
For every homogeneous ideal $I$ in $K[A_1,\dots,A_n]$, define the \emph{variety} of $I$
\[
\V(I) = \{ [a_1:\cdots:a_n] \in \mathbb{P}(K^E)\,|\, f(a)=0 \text{ for all homogeneous }f\in I\}.
\]
Often the term variety is reserved for certain types of $I$ or possibly certain types of fields $K$,
but we impose no additional conditions for our usage here.

The configuration hypersurface of a configuration $W$ is $\V\left( \Psi_W \right)$.
If the dimension of $W$ is $m$, then the configuration hypersurface $X_W$ is the degeneracy
locus $D_{m-1}(B_E|_W)\subset \mathbb{P}(K^E)$.
Recall that a degeneracy locus of a family of matrices $M(A)$ in $\mathbb{P}(K^E)$
is the set of points
\[
D_k(M) = \left\{ a\in \mathbb{P}(K^E) \, |\, \rank M(a) \leq k \right\}.
\]
There is a chain of degeneracy loci
\begin{align*}
&D_0(B_E|_W)\subseteq D_1(B_E|_W)\subseteq \dots \subseteq\\ 
&\qquad \qquad \subseteq D_{m-2}(B_E|_W) 
 \subseteq D_{m-1}(B_E|_W)\subseteq D_m(B_E|_W) = \mathbb{P}(K^E).
\end{align*}

The central goal of the section, Theorem~\ref{T:singularityRank}, 
identifies $D_k(B_E|W)$ for $0\leq k\leq m-2$ as
$\V(I_k)$ where
\[
I_k = \langle \partial_F \Psi_W \,|\, \abs{F}\leq k \rangle.
\]
Our tools are Lemma~\ref{L:lowRankSubspace} 
and the description of $I_k$ from Corollary~\ref{C:partialsOfConfigurationPolynomials}.
We do not use or study the \emph{ideal associated to a subset} $V$ of $\mathbb{P}(K^E)$
\[
\I(V) = \langle f\in K[A_1,\dots,A_n]\, |\, f \text{is homogeneous and } f(a)=0 \text{ for all }a\in V\rangle.
\]
Therefore, we avoid many algebraic geometry issues that arise over non-algebraically closed fields.
We do not attempt to identify $\I\left( D_k(B_E|W) \right)$ with $I_k$,
and we do not know whether $I_k$ provides a reduced scheme structure.

\begin{definition}[Order of Singularity Ideal]
Let $X$ be a projective variety defined by a homogeneous ideal $I$ 
in $K[A_1,\dots, A_n]$.
The \emph{first order singularity ideal}, $S(I)$ or $S(X)$, is the ideal generated by the set
\[
\left\{ \frac{\partial f}{\partial A_i} \, \Bigg|\, f\in I \text{ and } i=1,\dots, n \right\}.
\]
The \emph{$k$-th order singularity ideal} is $S^{(k)}(X)=S^{(k)}(I)=S(S^{(k-1)}(I))$.
The \emph{locus of order at least $k$ singularities} of $X$,
$\Sing_{\geq k} X$, is the
scheme defined by $S^{(k)}(I)$. 
\end{definition}
Note that when $X$ is a hypersurface with $I=\langle f\rangle $, the first singularity ideal
$S(X)$ defines the singular locus $\Sing X$ because the matrix of derivatives of $f$
is zero if and only if its rank is zero.
However,
$\Sing X$ will not be a hypersurface, so generally,
\[
\V(S^{(2)}(I)) \varsubsetneq \Sing(\Sing X).
\]
In terms of the Jacobian matrix for $\Sing X$, $ \Sing(\Sing X)$ is defined by a rank condition, but $\V(S^{(2)}(I))$ corresponds
to the Jacobian matrix being identically $0$. 
Rather than defining singular loci of singular loci, the $k$-th order singularity ideals filter $X$ by
the degree of the tangent cones. More details on this interpretation are provided in Section~\ref{S:tangentSpaces}.

Euler's formula for homogeneous polynomials states that
\begin{equation} \label{eq:Eulersformula}
\sum_{i=1}^n A_i \frac{\partial f}{\partial A_i} = \deg(f) f.
\end{equation}
Therefore, each homogeneous $f$ in $I$ is also in $S(I)$ using the assumption that the characteristic of $K$ is zero so that $\deg(f)$
is invertible. 
The ideal $I$ is homogeneous, 
so every $g\in I$ is also in $S(I)$
by applying the same reasoning  to its homogeneous pieces.
Therefore,
the order of singularity ideals form a chain
\[
I \subseteq S(I)\subseteq S^{(2)}(I) \subseteq \dots,
\]
and so do the order at least $k$ singular loci defined by these ideals
\[
X \supseteq \Sing_{\geq 1} X \supseteq \Sing_{\geq 2} X \supseteq \dots.
\]
In particular, note that the $k$-th order singularity ideal can be defined by
\[
S^{(k)}(I)=
\left\langle 
\left. \frac{\partial^m f}{\partial A_1^{j_1}\cdots \partial A_n^{j_n}} \, \right|\, f\in I,\, m\leq k,
\text{ and } \sum_{i=1}^n j_i=m
\right\rangle.
\]
Define \emph{the points of multiplicity $k$}
\[
\Mult_k X = \Sing_{\geq k-1}X - \Sing_{\geq k}X.
\]
In other words, $\Mult_k X$ is the set of points in $X$ whose tangent cones have degree $k$.
\begin{lemma}\label{L:finiteGeneratingSet}
If $I$ is generated by  homogeneous polynomials $f_1,\dots,f_k$, then $S(I)$ is generated by \[L=\left\{ \left. \frac{\partial f_j}{\partial A_i}\, \right| \, \text{all }i,j\right\}.\]
\end{lemma}
\begin{proof}
The set $L$ is contained in $S(I)$ by definition. 
Note that each $f_j$ is also in the ideal generated by $L$ by Euler's formula~\eqref{eq:Eulersformula}.
Consider a homogeneous element $f \in I$.
Then there are homogeneous polynomials $h_1,\dots, h_k$ such that
\[
f = \sum_{j=1}^k h_j f_j.
\]
Differentiating gives
\[
\frac{\partial f}{\partial A_i} = \sum_{j=1}^k 
\left(
\frac{\partial h_j}{\partial A_i}f_j +
h_j \frac{\partial f_j}{\partial A_i}
\right).
\]
Using Euler's formula~\eqref{eq:Eulersformula}, both terms on the right are in the ideal generated by $L$.
Nonhomogeneous $f\in I$ may be written as the sum of their homogeneous pieces, each in $I$ by
the homogeneity of the generators of $I$.
\end{proof}

Starting with one generator $f$ in Lemma~\ref{L:finiteGeneratingSet} and proceeding by induction
gives the following corollary.
\begin{corollary}\label{C:principalSingular}
For a principal ideal $I=\langle f \rangle$,
\[
S^{(k)}(I)=
\left\langle 
\left. \frac{\partial^m f}{\partial A_1^{j_1}\cdots \partial A_n^{j_n}} \, \right|\, m\leq k,
\text{ and } \sum_{i=1}^n j_i=m
\right\rangle.
\]
\end{corollary}

The main result of this section is
\begin{theorem}[Singular Loci are Degeneracy Loci] \label{T:singularityRank}
Let $W$ be a nonzero configuration of dimension $m$ in $K^E$. 
Then
\[
\Sing_{\geq k} X_W = D_{m-k-1}(B_E|_W),
\]
for $1\leq k \leq m-1$.
In particular,
\[
\Mult_k X_W = \Corank_k B_E|_W.
\]
\end{theorem}

%
\begin{lemma}[Singularities Determined by Restricted Configurations] \label{L:singularityRestrictedConfiguration}
Let $W$ be a nonzero configuration of dimension $m$ in $K^E$.
For $1\leq k \leq m-1$,
\[
\Sing_{\geq k} X_W = \V\left(\Big\langle \Psi_{W^{E-F}} \,\Big|\,  \abs{F}\leq k \Big\rangle \right) = \bigcap_{
\substack{ F \subseteq E \\ \abs{F}\leq k} }
X_{(W^{E-F},K^E)}.
\]
\end{lemma}
\begin{proof}
Note that when $\Psi_W$ is differentiated more than once with respect to the same variable, the derivative
is identically zero because no variable occurs more than once in each monomial.
Therefore, 
\[
S^{(k)}(X_W) = \langle \partial_F\Psi_W \,|\, \abs{F}\leq k\rangle
\]
using Corollary~\ref{C:principalSingular}.
By Corollary~\ref{C:partialsOfConfigurationPolynomials} with $H=E-F$,
\[
\Big\langle \partial_F\Psi_W \,\Big|\, \abs{F}\leq k\Big\rangle
=
\Big\langle \Psi_{W^{E-F}} \,\Big|\,  \abs{F}\leq k \Big\rangle,
\]
which completes the proof.
\end{proof}
\begin{proof}[Proof of Theorem~\ref{T:singularityRank}]
%
%
If $B_E(a)$ has rank at most $m-k-1$ on $W$, then it has rank at most $m-k-1$ restricted to the subspace $W^{E-F}$.
By Lemma~\ref{L:intersectionDimensionBound}, $\dim W^{E-F}$ is bounded below by $m-k$,
so $B_E(a)$ is not full rank when restricted to $W^{E-F}$. Therefore, the determinant vanishes:
\[\det B_E(a)|_{W^{E-F}}= \Psi_{W^{E-F}}(a)=0.\] 
 
If $a\in \V\left(\left\langle \Psi_{W^{E-F}} \, \big|\, \abs{F}\leq k \right\rangle \right)$, then $B_E(a)$ is degenerate when
restricted to all subspaces $W^{E-F}$ for $\abs{F}\leq k$. 
Note that the subspaces $W^{E-F}$ are intersections of the complete set of hyperplanes $\left\{ W^{E-e}\, |\, e\in E,\, W\neq W^{E-e}\right\}$.
As a result, 
the conditions of Lemma~\ref{L:lowRankSubspace} are satisfied, so the rank of $B_E(a)$ is at most $m-k-1$.

The statement about multiplicity and corank follows immediately.
\end{proof}
\begin{example}[Trivial Configuration]
The configuration $W=K^E$ has configuration polynomial $\Psi_W(A)=\prod_{e\in E} A_e$.
Its configuration hypersurface is the union of the coordinate hyperplanes.
For each $k$ between $1$ and $m-1$, $\Sing_{\geq k} X_W$ is the union of the coordinate linear subspaces of codimension $k+1$.
In particular, $\Sing_{\geq k}$ has codimension $1$ in $\Sing_{\geq k-1}$. For example in $\mathbb{P}(K^3 )$, 
there are three distinct coordinate hyperplanes 
whose union is $X_W$.  These hyperplanes meet at $\Sing_{\geq 1} X_W$, which consists of the three projective points representing the
coordinate axes.
The locus $\Sing_{\geq 2} X_W = \V( A_1, A_2, A_3)$ is empty in $\mathbb{P}(K^3)$, 
though in the affine space $K^3$ it is simply the origin.
 \end{example}
\begin{remark} 
The main theorems in~\cite{IlicLandsberg} and~\cite{Graham} provide conditions under which
degeneracy loci are nonempty.
The conditions of these theorems are satisfied for all $k$ between $1$ and $m-1$ for $D_k(B_E|_W)$, and
hence, $D_k(B_E|_W)=\Sing_{\geq m-k-1}X_W$ is nonempty for such $k$.
The connectedness of these degeneracy loci is a more delicate matter.
For example, whether the connectedness theorems in~\cite{Tu} and~\cite{HarrisTu} apply will depend on the specific values of $n$, $m$, and $k$.
\end{remark}
 
\subsection{Tangent Cones}\label{S:tangentSpaces}

In this section, we show how to interpret the locus of order at least $k$ singularities in terms of tangent cones.
The theory of tangent cones is greatly simplified by restricting the field $K$ to be algebraically closed,
so we assume $K$ is algebraically closed for this section.

Let $X$ be a projective hypersurface defined by a homogeneous polynomial $f$.
If $a=[a_1:\cdots:a_n]$ is a point in $X$ and $a_i\neq 0$, then the multivariate
Taylor formula in the local coordinates with $a_i=1$ expresses $f$ near $a$ as
\[
f\left(\frac{A}{A_i}\right) = \sum_{j=0}^{\deg f} \sum_{\abs{J}=j} \frac{1}{J!} \frac{\partial^J f}{\partial A^J}\Big|_{\frac{A}{A_i}=\frac{a}{a_i}} \left(\frac{A}{A_i}-\frac{a}{a_i}\right)^J.
\]
Therefore, when $a\in \Mult_k X$, all the terms with $j< k$ are zero and
\begin{equation}\label{eq:kTaylor}
f\left(\frac{A}{A_i}\right) = \sum_{j=k}^{\deg f} \sum_{\abs{J}=j} \frac{1}{J!} \frac{\partial^J f}{\partial A^J}\Big|_{\frac{A}{A_i}=\frac{a}{a_i}} \left(\frac{A}{A_i}-\frac{a}{a_i}\right)^J.
\end{equation}
Picking out the $k$-th term of Formula~\eqref{eq:kTaylor}
defines the affine tangent cone
\begin{equation} \label{eq:tangentCone}
TC_{a} \V(f) = \V\left( \sum_{\abs{J}=k} \frac{1}{J!} \frac{\partial^J f}{\partial A^J}\Big|_{\frac{A}{A_i}=\frac{a}{a_i}} \left( \frac{A}{A_i}-\frac{a}{a_i} \right)^J \right).
\end{equation}
One may homogenize the leading term at $a$
\begin{align*}
f_{a,k}\left(\frac{A}{A_i}\right) &= \sum_{\abs{J}=k} \frac{1}{J!} \frac{\partial^J f}{\partial A^J}\Big|_{\frac{A}{A_i}=\frac{a}{a_i}} \left(\frac{A}{A_i}-\frac{a}{a_i}\right)^J \\
&= \frac{1}{a_i^k A_i^k}\sum_{\abs{J}=k} \frac{1}{J!} \frac{\partial^J f}{\partial A^J}\Big|_{\frac{A}{A_i}=\frac{a}{a_i}} (a_i A-aA_i)^J,
\end{align*}
so the  projective tangent cone to $\V(f)$ at $a$ is 
\begin{equation}\label{eq:projectiveTangentCone}
\mathbb{T}C_{a} \V(f) = \V\left( \sum_{\abs{J}=k} \frac{1}{J!} \frac{\partial^J f}{\partial A^J}\Big|_{\frac{A}{A_i}=\frac{a}{a_i}} \left(a_i A-a A_i\right)^J \right).
\end{equation}
The notation $\left( a_i A- a A_i \right)^J$ for a tuple $J=(j_1,\dots,j_n)$ is shorthand for
\[
\left( a_i A-a A_i \right)^J = \prod_{\ell=1}^n \left( a_i A_{\ell} -a_{\ell} A_i \right)^{j_{\ell}}.
\]
In particular, if $j_i>0$, the product is zero.
\begin{proposition}[Configuration Tangent Cone]\label{P:configurationTangentCone}
Let $W$ be a nonzero configuration in $K^E$. 
If $a=[a_1:\cdots:a_n]$ is a point of multiplicity $k$ with $a_i\neq 0$, then the tangent cone to $X_W$ at $a$ is
\[
TC_a X_W = \V\left( \sum_{\substack{J\subseteq E, \,  \abs{J}=k\\ \dim W^{E-J} = \dim W-k}} C_J \Psi_{W^{E-J}}\left(\frac{a}{a_i}\right) \left( \frac{A}{A_i}-\frac{a}{a_i} \right)^J \right)
\]
where the $C_J$ are nonzero constants independent of $a$.
The projective tangent cone to $X_W$ at $a$ is
\[
\mathbb{T}C_a X_W = \V\left( \sum_{\substack{J\subseteq E, \,  \abs{J}=k\\ \dim W^{E-J} = \dim W-k}} C_J \Psi_{W^{E-J}}\left(\frac{a}{a_i}\right) \left(a_i A-a A_i\right)^J \right).
\]
If $\Psi_W$ is a graph polynomial, all $C_J=1$.
\end{proposition}
\begin{proof}
The proposition follows from Equations~\eqref{eq:tangentCone} and~\eqref{eq:projectiveTangentCone} setting $f=\Psi_W$.
Namely, the $k$-th order term of $\Psi_W$ at $a$ is
\[
\Psi_{W,k,a}\left(\frac{A}{A_i}\right) = \sum_{\abs{J}=k} \frac{1}{J!} \frac{\partial^J \Psi_W}{\partial A^J} \Big|_{\frac{A}{A_i}=\frac{a}{a_i}} \left(\frac{A}{A_i}-\frac{a}{a_i}\right)^J.
\]
The
configuration polynomial has degree at most one in each variable (Proposition~\ref{P:configurationPolynomial}),
so only tuples $J$
that are sequences of $0$s and $1$s need to be included in the sum. Such tuples correspond to subsets of $E$.
If $\abs{J}=k$, 
then Lemma~\ref{L:nontrivialPartialMore} simplifies
$\partial_J \Psi_W$ to $0$  or to $C_J \Psi_{W^{E-J}}$  depending on whether
$\dim W^{E-J} = \dim W -k$.
\end{proof}
The relationship between the rank loci and the singular loci for configuration hypersurfaces
is similar to the relationship between the two loci for generic symmetric determinantal loci.
We make the relationship explicit in the following Corollary~\ref{C:intersectionOfTangentCones}.
For every vector space $W$,
the generic symmetric degeneracy locus $\mathfrak{X}$ 
is the locus of points in $\mathbb{P}\left( \Sym^2 W\spcheck \right)$
that do not have full rank as a bilinear form on $W$. 
Suppose that $W$ is a configuration in the based vector space $K^E$ so that
there is a surjective restriction map
\[
\Sym^2 \left( K^E \right)\spcheck \xrightarrow{\pi} \Sym^2 W\spcheck.
\]
Let $Z$ be the kernel of $\pi$, let $\pi$ also denote the
rational map
\[
\mathbb{P}\left( \Sym^2\left(K^E\right)\spcheck \right) \dashrightarrow \mathbb{P}\left( \Sym^2 W\spcheck \right)
\]
defined on the complement of $\mathbb{P}(Z)$.
Because the fibers of $\pi$ are linear, 
$\mathbb{P}\left( \Sym^2\left(K^E\right)\spcheck \right)$
is a cone over $\mathbb{P}\left( \Sym^2 W\spcheck \right)$ with
vertex $\mathbb{P}(Z)$.

The family $\mathbb{P}\left(B_E(A)\right)\cong \mathbb{P}\left(K^E\right)$ is a 
linear subspace of $\mathbb{P}\left( \Sym^2\left(K^E\right)\spcheck\right)$
containing the configuration hypersurface $X_W$.
Let $L=\mathbb{P}(\pi(B_E(A)))$ be the image of the family $\mathbb{P}(B_E(A))$ 
in $\mathbb{P}(\Sym^2 W\spcheck)$.
Note that $X_W$ is a cone over $L\cap \mathfrak{X}$ with vertex $V=\mathbb{P}(B_E(A))\cap\mathbb{P}(Z)$;
that is, $\pi$ is defined on $X_W-V$ with image $L\cap \mathfrak{X}$ and linear fibers.
\begin{corollary}[Intersection of Tangent Cones] \label{C:intersectionOfTangentCones}
For every $x\in X_W-V$, 
\[
\Mult_x X_W = \Mult_{\pi(x)} \mathfrak{X},
\]
and therefore,
\[
\mathbb{T}C_{\pi(x)} \left( L\cap \mathfrak{X} \right) = L\cap \mathbb{T}C_{\pi(x)} \mathfrak{X},
\]
and similarly for the affine tangent cones.
When the vertex $V$ is empty, $\pi$ embeds $\mathbb{P}(B_E(A))$ into $\mathbb{P}(\Sym^2 W\spcheck)$, and
\[
\mathbb{T}C_x X_W = L\cap \mathbb{T}C_x \mathfrak{X}.
\]
\end{corollary}
\begin{proof}
When $\mathfrak{X}$ is cut by the linear section $L$, the leading term at $\pi(x)$ may vanish leading to a
higher multiplicity. 
It is well-known that $\Mult_{\pi(x)} \mathfrak{X}$ is the corank of $x$ as a bilinear
form on $W$ (see~\cite{JozefiakLascouxPragacz} or~\cite{Harris} for example).
By  Theorem~\ref{T:singularityRank},
$\Mult_x X_W$ is also  the corank of $x$ as a bilinear form on $W$.
Moreover, $\Mult_x X_W = \Mult_{\pi(x)} L\cap \mathfrak{X}$ because $X_W$ is 
a cone over $L\cap \mathfrak{X}$.
In other words, the
multiplicity at $\pi(x)$ does not increase  when restricting to the linear section $L\cap\mathfrak{X}$. 
Therefore, the leading term defining $X_W$ at $x$ is found by linear substitution of
$L$ into the leading term defining $\mathfrak{X}$ at $\pi(x)$, which means that the tangent cone of $L\cap \mathfrak{X}$
is just the linear section of $\mathbb{T}C_{\pi(X)} \mathfrak{X}$.
\end{proof}
\begin{figure}
\begin{center}
\begin{tikzpicture}[node distance = 3cm]
\SetVertexNoLabel
\GraphInit[vstyle=Classic]
\Vertex{a}
\SO(a){b}
\Edge[style={bend left,out=90, in=90},label={$e_1$},labelstyle={left}](b)(a)
\Edge[style={bend left},label={$e_2$},labelstyle={right}](b)(a)
\Edge[style={bend right},label={$e_3$},labelstyle={right}](b)(a)
\Edge[style={bend right,out=-90, in=-90},label={$e_4$},labelstyle={right}](b)(a)
\end{tikzpicture}
\end{center}
\caption{A simple example of a graph.}
\label{F:ThreeBananas}
\end{figure}
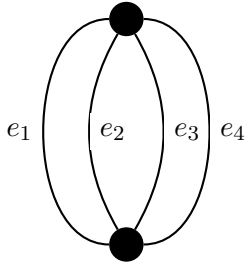
\begin{example}\label{E:tangentConeThreeBananas}
Consider the graph $G$ in Figure~\ref{F:ThreeBananas}.
Let $E=\{e_1,e_2,e_3,e_4\}$, and consider the configuration $W=H_1(G)$ in $K^E$.
There are orientations of $G$ for which an integral basis of $H_1(G)$ is given
by $\{ \ell_1,\ell_2,\ell_3\}$ where 
\begin{align*}
\ell_1 &= e_1+e_2\\
\ell_2 &= e_2+e_3\\
\ell_3 &= e_3+e_4.
\end{align*}
Elements of $\Sym^2 \left(K^E\right)\spcheck$ and $\Sym^2 W\spcheck$
can be represented by
symmetric matrices in these bases. Because these matrices are symmetric,
we only write the upper triangle to reduce clutter.
For $a\in \Sym^2 \left(K^E\right)\spcheck$ represented by
\[
a=
\begin{pmatrix}
a_{1} & a_{12} & a_{13} & a_{14}\\
* & a_{2} & a_{23} & a_{24}\\
* & * & a_{3} & a_{34}\\
* & * & * & a_{4}
\end{pmatrix},
\]
the image of $a$ in $\Sym^2 W\spcheck$ can be represented 
\[
\pi(a) =
\begin{pmatrix}
a_{1}+2a_{12}+a_{2} & a_{12}+a_{13}+a_{2}+a_{23} & a_{13}+a_{14}+a_{23}+a_{24} \\
* & a_{2}+2a_{23}+a_{3} & a_{23}+a_{24}+a_{3}+a_{34} \\
* & * & a_{3}+2a_{34}+a_{4}
\end{pmatrix}.
\]
%
An element of the  family $B_E(A)$ is of the form
\[
\begin{pmatrix}
a_{1} & 0 & 0 & 0\\
0 & a_{2} & 0 & 0\\
0 & 0 & a_{3} & 0 \\
0 & 0 & 0 & a_{4}
\end{pmatrix},
\]
and its image in $\pi\left(B_E(A)\right)$ is
\[
\begin{pmatrix}
a_{1}+a_{2} & a_{2} & 0 \\
* & a_{2}+a_{3} & a_{3} \\
* & * &  a_{3}+a_{4}
\end{pmatrix}.
\]

Let a generic element of $\Sym^2 W\spcheck$ be denoted
\[
b =
\begin{pmatrix}
b_1 & b_{12} & b_{13} \\
* & b_2 & b_{23} \\
* & * & b_3
\end{pmatrix}.
\]
The image $L=\mathbb{P}\left(\pi\left( B_E(A) \right)\right)$ is defined
by the ideal \[
I=\langle B_{13},B_2-B_{12}-B_{23}\rangle.
\]
The generic symmetric degeneracy locus is defined by 
the determinant 
\[
f(B)=B_1B_2B_3
+2B_{12}B_{23}B_{13}
-B_1B_{23}^2
-B_2B_{13}^2
-B_3B_{12}^2.
\]
The kernel of $\pi$ meets $B_E(A)$ at zero,
and thus the vertex of the map of projective
varieties is empty.
In particular, the isomorphism of $\mathbb{P}\left(B_E(A)\right)$ with
$L$ corresponds to the isomorphism of rings
\begin{align*}
K[B]/I &\stackrel{\pi^*}{\longrightarrow} K[A]\\
B_1 &\longmapsto A_1+A_2 \\
B_2 & \longmapsto A_2+A_3 \\
B_3 & \longmapsto A_3+A_4 \\
B_{12} & \longmapsto A_2 \\
B_{23} & \longmapsto A_3.
\end{align*}

Consider the rank-one matrix
\[
b = 
\begin{pmatrix}
1 & 0 & 0 \\
0 & 0 & 0 \\
0 & 0 & 0 
\end{pmatrix},
\]
and note that $b=\pi(a)$ where
\[
a = 
\begin{pmatrix}
1 & 0 & 0 & 0\\
0 & 0 & 0 & 0\\
0 & 0 & 0 & 0\\
0 & 0 & 0 & 0
\end{pmatrix}.
\]
Because $b$ has rank one, it must have multiplicity two on both $\mathfrak{X}$
and $X_W$. 
Therefore, all of the first partials of $f$ with respect to the $B$ coordinates vanish,
but there are second partials that do not:
\begin{align*}
\frac{\partial^2 f}{\partial B_{2}\partial B_3}(b) &= b_1=1\\
\frac{\partial^2 f}{\partial B_{23}^2}(b) &= -2b_1=-2.
\end{align*}
Therefore, the formula for the tangent cone is
\[
\mathbb{T}C_b \mathfrak{X} = \V( B_2B_{3}-B_{23}^2 ).
\]
Intersecting the tangent cone with the subspace $L$ and transforming to the $A$ coordinates
gives 
\[
L\cap \mathbb{T}C_b \mathfrak{X} = \V( (A_2+A_3)(A_3+A_4)-A_3^2 ).
\]
This hyperplane agrees with the computation for 
$\mathbb{T}C_a X_W$.
Namely, the second partials of $\Psi_W$ that do not vanish at $a$ are
\begin{align*}
\frac{\partial^2 \Psi_W}{\partial A_2 \partial A_3}(a) &= a_1 =1 \\
\frac{\partial^2 \Psi_W}{\partial A_2 \partial A_4}(a) &= a_1 =1 \\
\frac{\partial^2 \Psi_W}{\partial A_3 \partial A_4}(a) &= a_1 =1, 
\end{align*}
which means the tangent cone is defined by 
\[
\mathbb{T}C_a X_W =
\V( A_2A_3+A_2A_4+A_3A_4 ).
\]

For other linear subspaces $\tilde{L}$, it is possible that there are $b\in \tilde{L}\cap\mathfrak{X}$ such that 
$\Mult_b \mathfrak{X} \neq \Mult_b \tilde{L}\cap \mathfrak{X}$. 
For example, consider $\tilde{L} = \langle B_3, B_{23}\rangle$ and
\[
b = 
\begin{pmatrix}
1 & 0 & 0 \\
0 & 1 & 0 \\
0 & 0 & 0
\end{pmatrix}.
\]
In this case, the only non-vanishing first partial at $b$ is
\[
\frac{\partial f}{\partial B_3}(b) = b_1b_2-b_{12} = 1,
\]
so the tangent cone is
\[
\mathbb{T}C_b \mathfrak{X} = \V( B_3 ),
\]
which contains $\tilde{L}$.
In the coordinate ring $K[B_1,B_2,B_{12},B_{13}]$ for $\tilde{L}$,
$\tilde{L}\cap \mathfrak{X}$ is defined by
\[
\bar{f}(B) = -B_2B_{13}^2.
\]
In particular, all first partials of $\bar{f}$ vanish at $b$,
and the multiplicity at $b$ has increased by intersecting with $\tilde{L}$.

\end{example}

\section{Conclusion}\label{S:conclusion} 

We have identified  formulas for the derivatives of
configuration polynomials as configuration polynomials 
for restrictions of the original configuration.
These formulas allow linear algebra of degenerate
bilinear forms to prove our main result relating
the multiplicities and ranks of the points in 
configuration hypersurfaces.
The multiplicity-rank relationship is the same as
for the generic symmetric determinantal loci, but
our approach is quite different.
A more geometric approach along the lines of the
proof for the generic case would be interesting
and may provide a better understanding of the
geometry of the incidence variety for the
graph hypersurface.
Ideally, we would like to use the generic case
to prove results upon restricting to configuration
hypersurfaces.
As these spaces may be highly singular and
we have multiplicity information, 
Stratified Morse Theory seems a
promising route to cohomological calculations.

The identification of the second graph polynomial
as a configuration polynomial allows our
results to apply to both polynomials in the 
parametric Feynman integral~\eqref{eq:FeynmanIntegral},
but we have only considered momenta with values in a field $K$.
In some cases of interest to physicists, 
non-scalar momenta are required,
and we have not pursued that here.
The case of quaternionic momenta is handled by
Block and Kreimer~\cite{BK10}.
The configurations defining the first and second graph
polynomials come from homology groups relative to
zero and one-dimensional subspaces of the momentum $K^{V,0}$.
Allowing higher dimensional subspaces defines higher
order graph polynomials, and it is an open problem
to find combinatorial descriptions of them.
The review~\cite{BW} suggests higher analogs to the
first two graph polynomials but from a different perspective.
Understanding  configuration polynomials for $H_1(G)$
and $H_1(G,p)$ 
as a version of the all-minor matrix-tree theorem may
provide the link between these two versions of
higher order graph polynomials.

\bibliographystyle{amsalpha}
\bibliography{bibliography}

\end{document}